\documentclass[12pt]{article}

\usepackage[utf8]{inputenc}

\usepackage{a4,amsmath,amsfonts,amsthm,latexsym,amssymb,graphicx}
\usepackage{fullpage}
\usepackage{tocloft}
\usepackage{bm}
\usepackage{enumerate}
\usepackage{enumitem}
\usepackage{bbm}
\usepackage{mathtools}
\usepackage{tikz}
\usepackage{tikz-cd}
\usepackage{pgf}
\usepackage{hhline}
\usepackage{float}
\usepackage{url}
\usepackage[symbol]{footmisc}

\DeclareMathSymbol{\Q}{\mathalpha}{AMSb}{"51}
\DeclareMathSymbol{\R}{\mathalpha}{AMSb}{"52}
\DeclareMathSymbol{\Z}{\mathalpha}{AMSb}{"5A}
\DeclareMathSymbol{\N}{\mathalpha}{AMSb}{"4E}
\DeclareMathSymbol{\C}{\mathalpha}{AMSb}{"43}

\newcommand{\Ann}{\ensuremath{\text{\rm Ann}}}
\newcommand{\Per}{\ensuremath{\text{\rm Per}}}
\newcommand{\ord}{\ensuremath{\text{\rm ord}}}
\newcommand{\supp}{\ensuremath{\text{\rm supp}}}
\newcommand{\Res}{\ensuremath{\text{\rm Res}}}
\newcommand{\conv}{\ensuremath{\text{\rm conv}}}
\newcommand{\fib}[2]{\ensuremath{\text{\rm fib}}_{#1}(#2)}
\newcommand{\A}{\ensuremath{\mathcal{A}}}

\newcommand{\Patt}[2]{\ensuremath{\mathcal{L}_{#2}(#1)}}
\newcommand{\Lang}[1]{\ensuremath{\mathcal{L}(#1)}}

\renewcommand{\vec}[1]{\mathbf{#1}}

\newtheorem*{theorem*}{Theorem}
\newtheorem*{corollary*}{Corollary}

\newtheorem{theorem}{Theorem}
\newtheorem{lemma}[theorem]{Lemma}
\newtheorem{corollary}[theorem]{Corollary}

\newtheorem*{claim*}{Claim}

\newtheorem{reformulation}{Reformulation of Theorem}

\theoremstyle{definition}

\newtheorem{example}[theorem]{Example}

\title{On forced periodicity of perfect colorings}
\author{Pyry Herva and Jarkko Kari}
\date{}

\begin{document}

\maketitle

\begin{abstract}
%The abstract should briefly summarize the contents of the paper in 150--250 words.
\noindent
We study forced periodicity of two-dimensional configurations under certain constraints and use an algebraic approach to multidimensional symbolic dynamics in which $d$-dimensional configurations and finite patterns are presented as formal power series and Laurent polynomials, respectively, in $d$ variables.
We consider perfect colorings that are configurations such that the number of points of a given color in the neighborhood of any point depends only on the color of the point for some fixed relative neighborhood,
and we show that by choosing the alphabet suitably any perfect coloring has a non-trivial annihilator, that is, there exists a Laurent polynomial whose formal product with the power series presenting the perfect coloring is zero.
Using known results we obtain a sufficient condition for forced periodicity of two-dimensional perfect colorings.
As corollaries of this result we get simple new proofs for known results of forced periodicity on the square and the triangular grids.
Moreover, we obtain a new result concerning forced periodicity of perfect colorings in the king grid.
We also consider perfect colorings of a particularly simple type: configurations that have low abelian complexity with respect to some shape, and we generalize a result that gives a sufficient condition for such configurations to be necessarily periodic.
Also, some algorithmic aspects are considered.
\end{abstract}

%\tableofcontents

\section{Introduction}

We say that a $d$-dimensional configuration $c \in \A^{\Z^d}$, that is, a coloring of the $d$-dimensional integer grid $\Z^d$ using colors from a finite set $\A$ is a perfect coloring with respect to some finite relative neighborhood $D\subseteq \Z^d$ if the number of any given color of $\A$ in the pattern $c|_{\vec{u} + D}$ depends only on the color $c(\vec{u})$ for any $\vec{u} \in \Z^d$.
There is a similar version of this definition for general graphs: a vertex coloring $\varphi \colon V \rightarrow \A$ of a graph $G=(V,E)$ with a finite set $\A$ of colors is a perfect coloring of radius $r$ if the number of any given color in the $r$-neighborhood of a vertex $u \in V$ depends only on the color $\varphi(u)$ of $u$ \cite{puzynina, puzynina2}.
More generally, the definition of perfect colorings is a special case of the definition of equitable partitions \cite{godsil}.

If $\varphi \colon V \rightarrow \{0,1\}$ is a binary 
vertex coloring of a graph $G=(V,E)$  then we can define a subset $C \subseteq V$ of the vertex set -- a code -- such that it contains all the vertices with color $1$.
If $\varphi$ is a perfect coloring of radius $r$, then the code $C$ has the property that the number of codewords of $C$ in the $r$-neighborhood of a vertex $u \in V$ is $a$ if $u \not \in C$ and $b$ if $u \in C$ for some fixed non-negative integers $a$ and $b$.
This kind of code is called a perfect $(r,b,a)$-covering or simply just a perfect multiple covering \cite{Axenovich, coveringcodes}.
This definition is related to domination in graphs and covering codes \cite{fundamentals-of-domination,coveringcodes}.
% Golomb-Welch... \cite{Golomb-Welch-conjecture} \cite{Golomb-Welch-survey}...

Let $D \subseteq \Z^d$ be a finite set and $\A$ a finite set of colors.
Two finite patterns $p, q \in \A^D$ are abelian equivalent if the number of occurrences of each symbol in $\A$ is the same in them.
The abelian complexity of a configuration $c \in \A^{\Z^d}$ with respect to a finite shape $D$ is the number of abelian equivalence classes of patterns of shape $D$ in $c$ \cite{puzynina3}.
We note that if $c \in \A^{\Z^d}$ is a perfect coloring with respect to $D$ and $|\A| = n$, then the abelian complexity of $c$ with respect to $D$ is at most $n$.
Abelian complexity is a widely studied concept in one-dimensional symbolic dynamics and combinatorics on words \cite{lothaire}.

% Perfect colorings of the grid $\Z^d$ can be expressed in terms of confgurations having restricted abelian complexity

% abelian complexity
% yleistys yksulotteisesta paljon tutkitusta lothaire \cite{lothaire}
% erikoistappaus perfect coloring

In this paper we study forced periodicity of two-dimensional perfect colorings, that is, we study conditions under which  all the colorings are necessarily periodic.
We give a general condition for forced periodicity.
As corollaries of this result we get new proofs for known results \cite{Axenovich, puzynina,puzynina2} concerning forced periodicity of perfect colorings in the square and the triangular grid and a new result for forced periodicity of perfect colorings in the king grid.
Moreover, we study two-dimensional configurations of low abelian complexity, that is, configurations that have abelian complexity 1 with respect to some shape: we generalize a statement of forced periodicity concerning this type of configurations.
We use an algebraic approach \cite{icalp} to multidimensional symbolic dynamics, {\it i.e.}, we present configurations as formal power series and finite patterns as Laurent polynomials.
This approach was developed to make progress in a famous open problem in symbolic dynamics -- Nivat's conjecture \cite{Nivat} -- concerning forced periodicity of two-dimensional configurations that have a sufficiently low number of $m \times n$ rectangular patterns for some $m,n$. The Nivat's conjecture thus claims a two-dimensional generalization of the Morse-Hedlund theorem \cite{morse-hedlund}.

This article is an extended version of the conference paper \cite{DLT} where we considered forced periodicity of perfect coverings, that is, perfect colorings with only two colors.

\subsection*{The structure of the paper}

We begin in Section \ref{basics} by introducing the basic concepts of symbolic dynamics, cellular automata and graphs, and defining perfect colorings formally.
In Section \ref{algebraic approach} we present the relevant algebraic concepts and the algebraic approach to multidimensional symbolic dynamics,
and
in Section \ref{line polynomial factors} we describe an algorithm to find the line polynomial factors of a given two-dimensional Laurent polynomial.
In Section \ref{perfect coverings} we consider forced periodicity of perfect coverings, {\it i.e.}, perfect colorings with only two colors and then in Section \ref{perfect colorings} we extend the results from the previous section to concern perfect colorings using arbitrarily large alphabets.
After this we prove a statement concerning forced periodicity of two-dimensional configurations of low abelian complexity in Section \ref{abelian}.
In Section \ref{algortihmic aspects} we consider some algorithmic questions concerning perfect colorings.

\section{Preliminaries} \label{basics}

%\subsection*{Configurations, periodicity, finite patterns and subshifts}

%\subsection*{Configurations and periodicity}

\subsection*{Basics on symbolic dynamics}

Let us review briefly some basic concepts of symbolic dynamics relevant to us.
For a reference see \emph{e.g.} \cite{tullio,kurka,lindmarcus}.
Although our results concern mostly two-dimensional configurations, we state our definitions %and some auxiliary results
in an arbitrary dimension.

Let $\A$ be a finite set (the \emph{alphabet}) and let $d$ be a positive integer (the \emph{dimension}).
A $d$-dimensional \emph{configuration} over $\A$ is a coloring of the infinite grid $\Z^d$ using colors from $\A$, that is, an element of $\A^{\Z^d}$ -- the \emph{$d$-dimensional configuration space} over the alphabet $\A$.
%which we call the $d$-dimensional \emph{configuration space}.
We denote by $c_{\vec{u}} = c(\vec{u})$ the symbol or color that a configuration $c \in \A^{\Z^d}$ has in cell $\vec{u}$.
The \emph{translation} $\tau^{\vec{t}}$ by a vector $\vec{t} \in \Z^d$ shifts a configuration $c$ such that $\tau^{\vec{t}}(c)_{\vec{u}} = c_{\vec{u} - \vec{t}}$ for all $\vec{u}\in\Z^d$.
A configuration $c$ is \emph{$\vec{t}$-periodic} if $\tau^{\vec{t}}(c) = c$,
and it is \emph{periodic} if it is $\vec{t}$-periodic for some non-zero $\vec{t} \in \Z^d$.
Moreover, we say that a configuration is \emph{periodic in direction} $\vec{v} \in \Q^d\setminus\{\vec{0}\}$ if it is $k \vec{v}$-periodic for some $k \in \Z$.
% In the two-dimensional setting $d=2$, if there are two linearly independent vectors of periodicity then $c$ is called \emph{two-periodic}.
A $d$-dimensional configuration $c$
is \emph{strongly periodic} if it has $d$
linearly independent vectors of periodicity.
A strongly periodic configuration is periodic in every rational direction.
Two-dimensional strongly periodic configurations are called \emph{two-periodic}.
%Two-dimensional periodic configurations that are not two-periodic are called \emph{one-periodic}.

%\subsection*{Finite patterns}

A finite \emph{pattern} is an assignment of symbols on some finite shape $D \subseteq \Z^d$, that is, an element of $\A^D$.
%where $\A$ is some fixed alphabet.
In particular, the finite patterns in $\A^D$ are called \emph{$D$-patterns}.
Let us denote by $\A^*$ the set of all finite patterns over $\A$ where the dimension $d$ is known from the context.
We say that a finite pattern $p \in \A^D$ \emph{appears} in a configuration $c \in \A^{\Z^d}$ or that $c$ \emph{contains} $p$ if $\tau^{\vec{t}}(c)|_D = p$ for some $\vec{t} \in \Z^d$.
For a fixed shape $D$, the set of all $D$-patterns of $c$ is the set $\Patt{c}{D} = \{ \tau^{\vec{t}}(c)|_D \mid \vec{t} \in \Z^d \}$ and the set of all finite patterns of $c$
is denoted by $\Lang{c}$ which is called the \emph{language of $c$}. For a set $\mathcal{S} \subseteq \A^{\Z^d}$ of configurations
we define $\Patt{\mathcal{S}}{D}$ and  $\Lang{\mathcal{S}}$ as the unions of $\Patt{c}{D}$ and $\Lang{c}$ over all
$c\in \mathcal{S}$, respectively.

The \emph{pattern complexity} $P(c,D)$ of a configuration $c \in \A^{\Z^d}$ with respect to a shape $D$ is the number of distinct $D$-patterns that $c$ contains.
%Let $\A = \{ a_1, \ldots , a_n \}$ be a finite alphabet.
%Let $\A = \{ a_1, \ldots , a_n \}$ be a finite alphabet and let $p$ be a finite pattern over $\A$.
For any $a \in \A$ we denote by $|p|_a$ the number of occurrences of the color $a$ in a finite pattern $p$.
Two finite patterns $p,q \in \A^{D}$ are called \emph{abelian equivalent} if $|p|_a = |q|_a$ for all $a \in \A$, that is, if the number of occurrences of each color is the same in both $p$ and $q$.
The \emph{abelian complexity} $A(c,D)$ of a configuration $c \in \A^{\Z^2}$ \emph{with respect to a finite shape $D$} is the number of different $D$-patterns in $c$ up to abelian equivalence \cite{puzynina3}.
Clearly $A(c,D) \leq P(c,D)$.
%The equality may or may not occur in above inequality, Example \ref{abelian example} illustrates this.
We say that $c$ has \emph{low complexity} with respect to $D$ if
$$
P(c,D) \leq |D|
$$
and that $c$ has \emph{low abelian complexity} with respect to $D$ if
$$
A(c,D) = 1.
$$

\iffalse
\begin{example} \label{abelian example}
    %Clearly, $A(c,D) \leq P(c,D)$ for any configuration $c$ and for any finite shape.
    Consider the configuration $c \in \{0,1\}^{\Z^2}$ defined as $c(i,j) = i+j$ (mod 2).
    See Figure \ref{checker board} where white color corresponds to 0 and black to 1.
    Let $D_1 = \{0,1\} \times \{0,1\}$ be the $2 \times 2$ square and $D_2 = \{0,1,2\} \times \{0,1,2\}$ the $3 \times 3$ square.
    Clearly, $P(c,D_1) = 2 = P(c,D_2)$.
    However, we have $A(c,D_1) = 1$ but $A(c,D_2) = 2$.
\end{example}

\begin{figure}[ht]
    \centering
    \begin{tikzpicture}[scale=0.5]
        \draw[] (0,0) grid (10,10);
        \foreach \x in {0,...,9} \foreach \y in {0,...,9}
    {
        \pgfmathparse{mod(\x+\y,2) ? "black" : "white"}
        \edef\colour{\pgfmathresult}
        \path[fill=\colour] (\x,\y) rectangle ++ (1,1);
    }
    \end{tikzpicture}
    \caption{A configuration that has...}
    \label{checker board}
\end{figure}
\fi

%\subsection*{Basic concepts of symbolic dynamics}

The configuration space $\A^{\Z^d}$ can be made a compact topological space by endowing $\A$ with the discrete topology and considering the product topology it induces on $\A^{\Z^d}$ -- the \emph{prodiscrete topology}.
This topology is
% metrizable and indeed it is
induced by a metric where two configurations are close if they agree on a large area around the origin.
So, $\A^{\Z^d}$ is a compact metric space.
\iffalse
The prodiscrete topology has a basis consisting of so-called \emph{cylinders}.
 that are sets of the form
$$
[p] = \{ c \in \A^{\Z^d} \mid c|_D = p \}
$$
where $p \in \A^D$ is a finite pattern with domain $D$.
\fi

A subset $\mathcal{S} \subseteq \A^{\Z^d}$ of the configuration space is a \emph{subshift} if it is topologically closed and translation-invariant meaning that if $c \in \mathcal{S}$, then for all $\vec{t} \in \Z^d$ also $\tau^{\vec{t}}(c) \in \mathcal{S}$.
Equivalently, subshifts can be defined using forbidden patterns:
Given a set $F \subseteq \A^*$ of \emph{forbidden} finite patterns, the set
$$
X_F=\{c \in \A^{\Z^d} \mid  \Lang{c} \cap F = \emptyset \}
$$
of configurations that avoid all forbidden patterns
is a subshift. Moreover, every subshift is obtained by forbidding some set of finite patterns.
If $F \subseteq \A^*$ is finite, then we say that $X_F$ is a \emph{subshift of finite type} (SFT).
\iffalse
Subshifts of finite type can also be defined using \emph{allowed} patterns.
Let $D\subseteq\mathbb{Z}^d$ be a finite shape and let
$P\subseteq A^D$ be a set of allowed $D$-patterns. The set
$$X^+_P=\{c\in A^{\Z^d}\ |\ \Patt{c}{D}\subseteq P \}$$
of configurations whose $D$-patterns are among $P$ is clearly the SFT obtained by forbidding the
patterns in the complement $F=A^D\setminus P$.
If $|P|\leq |D|$, then we say that  $X^+_P$
is a \emph{low complexity subshift}. Notice that as mentioned in the introduction the tiling problem can be reformulated using subshifts of finite type. It is equivalent with asking whether $X^+_P$ is non-empty when $P$ is the set of allowed $m {\times} n$ arrays of colors. Of course the tiling problem or the question of non-emptiness of an SFT is a valid question in any dimension.
\fi

%More generally, if the number of different $D$-patterns that appear in elements of a subshift  is at most $|D|$, the size of the shape, we say that the subshift has low complexity (with respect to $D$).

%We say a non-empty subshift is \emph{aperiodic} if it does not contain any periodic configurations. R. Berger was the first to show the surprising fact that there are two-dimensional aperiodic SFTs~\cite{berger}, thus initiating the theory of aperiodic tilings and quasicrystals. There are no one-dimensional aperiodic SFTs.

The \emph{orbit} of a configuration $c$ is the set $\mathcal{O}(c) = \{ \tau^{\vec{t}}(c) \mid \vec{t} \in \Z^d \}$ of its every translate.
The
\emph{orbit closure} $\overline{\mathcal{O}(c)}$ is the topological closure of its orbit under the prodiscrete topology.
The orbit closure of a configuration $c$ is the smallest subshift that contains $c$. It consists of all configurations $c'$ such that $\Lang{c'}\subseteq \Lang{c}$.
\iffalse
Configuration $c$ is called \emph{uniformly recurrent} if every configuration $c'$ in its orbit closure satisfies $\Lang{c'}= \Lang{c}$.
This is equivalent to the condition that the orbit closure of $c$ is a \emph{minimal} subshift, {\it i.e.}, no proper non-empty subset of $\overline{\mathcal{O}(c)}$ is a subshift.
A classical result by Birkhoff on minimal dynamical systems implies that every non-empty subshift contains a uniformly recurrent configuration~\cite{birkhoff}.
\fi

\subsection*{Cellular automata}

Let us describe briefly an old result of cellular automata theory 
that we use in Section~\ref{perfect colorings}.
See \cite{casurveyjarkko} for a more thorough survey on the topic.

A $d$-dimensional \emph{cellular automaton} or a \emph{CA} for short over a finite alphabet $\A$ is a map $F \colon \A^{\Z^d} \longrightarrow \A^{\Z^d}$ determined by a neighborhood vector $N = (\vec{t}_1, \ldots, \vec{t}_n )$ and a local rule $f \colon \A^n \longrightarrow \A$ such that
$$
F(c)(\vec{u}) = f(c(\vec{u} + \vec{t}_1), \ldots , c(\vec{u} + \vec{t}_n)).
$$
A CA is \emph{additive} or \emph{linear} if its local rule is of the form
$$
f(x_1,\ldots,x_n) = a_1x_1 + \ldots + a_n x_n
$$
where $a_1,\ldots,a_n \in R$ are elements of some finite ring $R$ and $\A$ is an $R$-module.

In Section~\ref{perfect colorings} we consider the surjectivity of cellular automata and use a classic result called the \emph{Garden-of-Eden theorem} proved by Moore and Myhil that gives a characterization for surjectivity in terms of injectivity on ``finite'' configurations.
Two configurations $c_1$ and $c_2$ are called \emph{asymptotic} if the set $\text{diff}(c_1,c_2) = \{ \vec{u} \mid c_1(\vec{u}) \neq c_2(\vec{u}) \}$ of cells where they differ is finite.
A cellular automaton $F$ is \emph{pre-injective} if $F(c_1) \neq F(c_2)$ for any distinct asymptotic configurations $c_1$ and $c_2$.
Clearly injective CA are pre-injective.
The Garden-of-Eden theorem states that pre-injectivity of a CA is equivalent to surjectivity:

\begin{theorem*}[Garden-of-Eden theorem, \cite{Moore1962,Myhill1963}]
    A CA is surjective if and only if it is pre-injective.
\end{theorem*}

\noindent
In the one-dimensional setting the Garden-of-Eden theorem yields the following corollary:

\begin{corollary*}
    For a one-dimensional surjective CA every configuration has only a finite number of pre-images.
\end{corollary*}

\subsection*{Graphs}

In this paper we consider graphs that are \emph{simple}, \emph{undirected} and \emph{connected}.
A graph $G$ that has vertex set $V$ and edge set $E$ is denoted by $G=(V,E)$.
The \emph{distance} $d(u,v)$ of two vertices $u\in V$ and $v\in V$ of a graph $G = (V,E)$ is the length of a shortest path between them in $G$.
The $r$-neighborhood of $u \in V$ in a graph $G = (V,E)$ is the set $N_r(u) = \{ v \in V \mid d(v,u) \leq r \}$.
The graphs we consider has vertex set $V = \Z^2$ and a translation invariant edge set $E \subseteq \{ \{ \vec{u}, \vec{v} \} \mid \vec{u}, \vec{v} \in \Z^2, \vec{u} \neq \vec{v} \}$.
This implies that for all $r$ and for any two points $\vec{u}\in \Z^2$ and $\vec{v} \in \Z^2$ their $r$-neighborhoods are the same up to translation, that is,
$N_r(\vec{u}) =N_r(\vec{v}) + \vec{u} - \vec{v}$.
Moreover, we assume that all the vertices of $G$ have only finitely many neighbors, {\it i.e.}, we assume that the \emph{degree} of $G$ is finite.
We call these graphs two-dimensional \emph{(infinite) grid graphs} or just \emph{(infinite) grids}.
%The \emph{$r$-neighborhood} $N_r(\vec{u}) = \{ \vec{v} \in \Z^2 \mid d(\vec{v} , \vec{u}) \}$ of a vertex $\vec{u} \in \Z^2$ in an infinite grid is determined by the $r$-neighborhood of $\vec{0}$.
%Indeed, if $N$ is the neighborhood of $\vec{0}$ then $\vec{u} + N$ is the neighborhood of $\vec{u}$ due to the translation invariance of infinite grid graphs.
In a grid graph $G$, let us call the $r$-neighborhood of $\vec{0}$ the \emph{relative $r$-neighborhood} of $G$ since it determines the $r$-neighborhood of any vertex in $G$.
Indeed, for all $\vec{u} \in \Z^2$ we have $N_r(\vec{u}) = N_r + \vec{u}$ where $N_r$ is the relative $r$-neighborhood of $G$.
Given the edge set of a grid graph, the relative $r$-neighborhood is determined for every $r$.
% Conversely, given a finite relative neighborhood $N \subseteq \Z^2$ we can define a grid graph such that two vertices are adjacent if one of them is in the neighborhood of the other.
We specify three 2-dimensional infinite grid graphs:

\begin{itemize}
	\item The \emph{square grid} is the infinite grid graph
    $%\mathcal{S} =
    (\Z^2, E_{\mathcal{S}})$ with
    $$E_{\mathcal{S}} = \{ \{ \vec{u} , \vec{v} \} \mid \vec{u} - \vec{v} \in \{  (\pm 1,0), (0,\pm 1) \} \}.$$
	\item The \emph{triangular grid} is the infinite grid graph $%\mathcal{T} =
    (\Z^2, E_{\mathcal{T}})$ with
    $$
    E_{\mathcal{T}} = \{ \{ \vec{u}, \vec{v} \} \mid \vec{u} - \vec{v} \in \{ (\pm 1,0),(0,\pm 1),(1,1),(-1,-1) \}  \}.
    $$
    \item The \emph{king grid} is the infinite grid graph
    $%\mathcal{K} =
    (\Z^2, E_{\mathcal{K}})$ with
    $$E_{\mathcal{K}} = \{ \{ \vec{u}, \vec{v} \} \mid \vec{u} - \vec{v} \in \{ (\pm 1,0),(0,\pm 1),(\pm 1,\pm 1) \}  \}.$$
\end{itemize}
The relative $2$-neighborhoods of these grid graphs are pictured in Figure \ref{neighborhoods}.

\begin{figure}
    \centering
    \begin{tikzpicture}[scale=0.5]
        \draw (0,0) grid (6,6);
        \draw (8,0) grid (14,6);
        \draw (16,0) grid (22,6);

        %the square grid
        \draw[fill=black] (3,3) circle(4pt);
        \draw[fill=black] (2,3) circle(4pt);
        \draw[fill=black] (1,3) circle(4pt);
        \draw[fill=black] (3,4) circle(4pt);
        \draw[fill=black] (3,5) circle(4pt);
        \draw[fill=black] (4,3) circle(4pt);
        \draw[fill=black] (5,3) circle(4pt);
        \draw[fill=black] (3,2) circle(4pt);
        \draw[fill=black] (3,1) circle(4pt);
        \draw[fill=black] (2,4) circle(4pt);
        \draw[fill=black] (4,4) circle(4pt);
        \draw[fill=black] (4,2) circle(4pt);
        \draw[fill=black] (2,2) circle(4pt);
        \draw[very thick] (0.5,3) -- (3,5.5) -- (5.5,3) -- (3,0.5) -- (0.5,3);
        %\node at (1.5,4.5) {$N_2$};

        %the triangular grid
        \draw[fill=black] (11,3) circle(4pt);
        \draw[fill=black] (2+8,3) circle(4pt);
        \draw[fill=black] (1+8,3) circle(4pt);
        \draw[fill=black] (3+8,4) circle(4pt);
        \draw[fill=black] (3+8,5) circle(4pt);
        \draw[fill=black] (4+8,3) circle(4pt);
        \draw[fill=black] (5+8,3) circle(4pt);
        \draw[fill=black] (3+8,2) circle(4pt);
        \draw[fill=black] (3+8,1) circle(4pt);
        \draw[fill=black] (2+8,4) circle(4pt);
        \draw[fill=black] (4+8,4) circle(4pt);
        \draw[fill=black] (4+8,2) circle(4pt);
        \draw[fill=black] (2+8,2) circle(4pt);
        \draw[fill=black] (2,3) circle(4pt);
        \draw[fill=black] (12,5) circle(4pt);
        \draw[fill=black] (13,5) circle(4pt);
        \draw[fill=black] (13,4) circle(4pt);
        \draw[fill=black] (10,1) circle(4pt);
        \draw[fill=black] (9,1) circle(4pt);
        \draw[fill=black] (9,2) circle(4pt);
        \draw[very thick] (8.5,0.5) -- (8.5,3.5) -- (11,5.5) -- (13.5,5.5) -- (13.5,2.5) -- (11,0.5) -- (8.5,0.5);

        %the king grid
        \draw[fill=black] (19,3) circle(4pt);
        \draw[fill=black] (2+8+8,3) circle(4pt);
        \draw[fill=black] (1+16,3) circle(4pt);
        \draw[fill=black] (3+16,4) circle(4pt);
        \draw[fill=black] (3+16,5) circle(4pt);
        \draw[fill=black] (4+8+8,3) circle(4pt);
        \draw[fill=black] (5+8+8,3) circle(4pt);
        \draw[fill=black] (3+8+8,2) circle(4pt);
        \draw[fill=black] (3+8+8,1) circle(4pt);
        \draw[fill=black] (2+8+8,4) circle(4pt);
        \draw[fill=black] (4+8+8,4) circle(4pt);
        \draw[fill=black] (4+8+8,2) circle(4pt);
        \draw[fill=black] (2+8+8,2) circle(4pt);
        \draw[fill=black] (2+8,3) circle(4pt);
        \draw[fill=black] (12+8,5) circle(4pt);
        \draw[fill=black] (13+8,5) circle(4pt);
        \draw[fill=black] (13+8,4) circle(4pt);
        \draw[fill=black] (10+8,1) circle(4pt);
        \draw[fill=black] (9+8,1) circle(4pt);
        \draw[fill=black] (9+8,2) circle(4pt);
        \draw[fill=black] (18,5) circle(4pt);
        \draw[fill=black] (17,5) circle(4pt);
        \draw[fill=black] (17,4) circle(4pt);
        \draw[fill=black] (20,1) circle(4pt);
        \draw[fill=black] (21,1) circle(4pt);
        \draw[fill=black] (21,2) circle(4pt);
        \draw[very thick] (16.5,0.5) -- (16.5,5.5) -- (21.5,5.5) -- (21.5,0.5) --(16.5,0.5);
    \end{tikzpicture}
    \caption{The relative $2$-neighborhoods of the square grid, the triangular grid and the king grid, respectively.}
    \label{neighborhoods}
\end{figure}
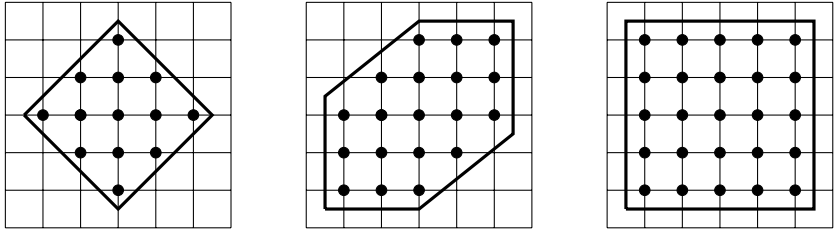

%A coloring of an infinite grid using colors from a finite set $\A$ is then an element of $\A^{\Z^2}$, that is, a two-dimensional configuration with  some neighborhood $N$.
%We study certain type of colorings of infinite grids which are elements of the configuration space $\A^{\Z^d}$ where $\A$ is the set of colors used in the coloring.

\subsection*{Perfect colorings}

Let $\A = \{a_1,\ldots,a_n\}$ be a finite alphabet of $n$ colors and let $D \subseteq \Z^d$ be a finite shape.
A configuration $c \in \A^{\Z^d}$ is a \emph{perfect coloring with respect to $D \subseteq \Z^d$} or a \emph{$D$-perfect coloring} if for all $i,j \in \{1,\ldots, n\}$ there exist numbers $b_{ij}$ such that for all $\vec{u} \in \Z^d$ with $c_{\vec{u}} = a_j$ the number of occurrences of color $a_i$ in the $D$-neighborhood of $\vec{u}$, {\it i.e.},
in the pattern $c|_{\vec{u} + D}$ is exactly $b_{ij}$.
The \emph{matrix of a $D$-perfect coloring} $c$ is the matrix $\mathbf{B} = (b_{ij})_{n \times n}$ where the numbers $b_{ij}$ are as above.
A $D$-perfect coloring with matrix $\mathbf{B}$ is called a (perfect) \emph{$(D,\vec{B})$-coloring}.
Any $D$-perfect coloring is called simply a perfect coloring.
In other words, a configuration is a perfect coloring if the number of cells of a given color in the given neighborhood of a vertex $\vec{u}$ depends only on the color of $\vec{u}$.
%There is a related concept for arbitrary graphs \cite{puzynina} and a more general concept equitable partitions \cite{godsil}...

Perfect colorings are defined also for arbitrary graphs $G=(V,E)$.
Again, let $\A = \{a_1, \ldots , a_n \}$ be a finite set of $n$ colors.
A vertex coloring $\varphi \colon V \rightarrow \A$ of $G$ is an $r$-perfect coloring with matrix $\mathbf{B} = (b_{ij})_{n \times n}$ if the number of vertices of color $a_i$ in the $r$-neighborhood of a vertex of color $a_j$ is exactly $b_{ij}$.
Clearly if $G$ is a translation invariant graph with vertex set $\Z^d$, then the $r$-perfect colorings of $G$ are exactly the $D$-perfect colorings in $\A^{\Z^d}$ where $D$ is the relative $r$-neighborhood of the graph $G$.
%Note that the definition of perfect colorings of graphs is a special case of the definition of equitable partitions of graphs \cite{godsil}.

%Perfect colorings in graphs... equitable partitions \cite{godsil}... centered functions... generalized centered functions...

\section{Algebraic concepts} \label{algebraic approach}

We review the basic concepts and some results relevant to us concerning an algebraic approach to multidimensional symbolic dynamics introduced and studied in \cite{icalp}. See also~\cite{surveyjarkko}
for a short survey of the topic.

%\subsection*{Configurations as integral power series}

%Assume that $\A \subseteq \Z$ is a finite subset of integers. We express a configuration $c \in \A^{\Z^d}$ as the formal power series
%$$
%c(X) = \sum_{\vec{u} \in \Z^d} c_{\vec{u}} X^{\vec{u}}
%$$
%in $d$ variables $x_1, \ldots , x_d$ where we have denoted $X = (x_1, \ldots , x_d)$ and $X^{\vec{u}} = x_1^{u_1} \cdots x_d^{u_d}$ for any $\vec{u} = (u_1,\ldots ,u_d) \in \Z^d$.
%For $d=2$ we usually denote $X=(x,y)$.
%More generally
Let $c \in \A^{\Z^d}$ be a $d$-dimensional configuration.
The power series presenting $c$ is the
formal power series
$$
c(X) = c(x_1,\ldots,x_d) = \sum_{\vec{u} = (u_1,\ldots,u_d) \in \Z^d} c_{\vec{u}} x_1^{u_1} \cdots x_d^{u_d} = \sum_{\vec{u} \in \Z^d} c_{\vec{u}} X^{\vec{u}}
$$
in $d$ variables $X = (x_1, \ldots , x_d)$.
We denote the set of all formal power series in $d$ variables $X = (x_1,\ldots ,x_d)$
over a domain $M$ by $M[[X^{\pm1}]] = M[[x_1^{\pm1},\ldots,x_d^{\pm1}]]$.
If $d=1$ or $d=2$, we denote $x=x_1$ and $y=x_2$.
A power series is \emph{finitary} if it has only finitely many distinct coefficients and \emph{integral} if its coefficients are all integers, {\it i.e.}, if it belongs to the set $\Z[[X^{\pm1}]]$.
A configuration is always presented by a finitary power series and a finitary power series always presents a configuration. So, from now on we may call any finitary power series a configuration.

We consider also Laurent polynomials which we may call simply just polynomials.
We denote the set of Laurent polynomials in $d$ variables $X=(x_1,\ldots,x_d)$ over a ring $R$ by $R[X^{\pm1}] = R[x_1^{\pm1}, \ldots, x_d^{\pm1}]$.
The term ``proper'' is used when we talk about proper ({\it i.e.}, non-Laurent) polynomials and denote the proper polynomial ring over $R$ by $R[X]$ as usual.

We say that two Laurent polynomials have no common factors if all their common factors are units in the polynomial ring under consideration and that they have a common factor if they have a non--unit common factor.
For example, in $\C[X^{\pm1}]$ two polynomials have no common factors if all their common factors are constants or monomials,
and
two proper polynomials in $\C[X]$
have no common factors if all their common factors are constants.
%Note that we may denote a power series $c(X)$ and polynomial $f(X)$ also by $c$ and $f$, respectively, for short.
The \emph{support} of a power series $c = c(X) = \sum_{\vec{u}  \in \Z^d} c_{\vec{u}} X^{\vec{u}}$ is the set $\supp(c) = \{ \vec{u} \in \Z^d \mid c_{\vec{u}} \neq 0 \}$.
Clearly a polynomial is a power series with a finite support.
The $k$th dilation of a polynomial $f(X)$ is the polynomial $f(X^k)$.
See Figure \ref{dilation} for an illustration of dilations.

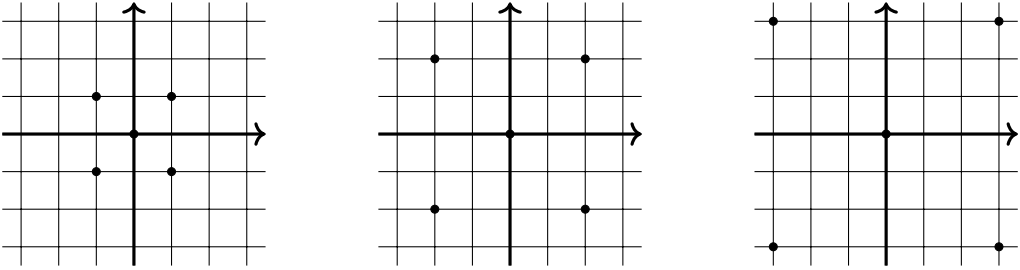
\begin{figure}[ht]
    \centering
    \begin{tikzpicture}[scale=0.5]
        \draw[fill=black] (-1,1) circle(3pt);
        \draw[fill=black] (-1,-1) circle(3pt);
        \draw[fill=black] (1,1) circle(3pt);
        \draw[fill=black] (1,-1) circle(3pt);
        \draw (16.5-20,-3.5) grid (23.5-20,3.5);
        \draw[very thick,->] (0,-3.5) -- (0,3.5);
        \draw[very thick,->] (-3.5,0) -- (3.5,0);
        \draw[fill=black] (0,0) circle(3pt);

        \draw[fill=black] (-2+10,2) circle(3pt);
        \draw[fill=black] (-2+10,-2) circle(3pt);
        \draw[fill=black] (2+10,2) circle(3pt);
        \draw[fill=black] (2+10,-2) circle(3pt);
        \draw (16.5-10,-3.5) grid (23.5-10,3.5);
        \draw[very thick,->] (0+10,-3.5) -- (0+10,3.5);
        \draw[very thick,->] (-3.5+10,0) -- (3.5+10,0);
        \draw[fill=black] (10,0) circle(3pt);

        \draw[fill=black] (-3+20,3) circle(3pt);
        \draw[fill=black] (-3+20,-3) circle(3pt);
        \draw[fill=black] (3+20,3) circle(3pt);
        \draw[fill=black] (3+20,-3) circle(3pt);
        \draw (16.5,-3.5) grid (23.5,3.5);
        \draw[very thick,->] (0+20,-3.5) -- (0+20,3.5);
        \draw[very thick,->] (-3.5+20,0) -- (3.5+20,0);
        \draw[fill=black] (20,0) circle(3pt);
    \end{tikzpicture}
    \caption{The supports of the polynomial $f(X) = 1 + x^{-1}y^{-1} + x^{-1}y^{1} + x^{1}y^{-1} + x^{1}y^{1}$ and its dilations $f(X^2)$ and $f(X^3)$.}
    \label{dilation}
\end{figure}

The $x_i$-resultant $\Res_{x_i}(f,g)$ of two proper polynomials $f,g \in R[x_1, \ldots , x_d]$
is
the determinant of the \emph{Sylvester matrix} of $f$ and $g$ with respect to variable $x_i$.
We omit the details
which the reader can check from \cite{cox}, and instead we consider the resultant $\Res_{x_i}(f,g) \in R[x_1,\ldots,x_{i-1},x_{i+1},\ldots,x_d]$ for every $i \in \{1,\ldots,d\}$
as a certain proper polynomial that has the following two properties:

\begin{itemize}
    \item $\Res_{x_i}(f,g)$ is in the ideal generated by $f$ and $g$, {\it i.e.}, there exist proper polynomials $h$ and $l$ such that
    $$
    hf + lg = \Res_{x_i}(f,g).
    $$
    \item If two proper polynomials $f$ and $g$ have no common factors in $R[x_1, \ldots , x_d]$, then $\Res_{x_i}(f,g) \neq 0$.
\end{itemize}

Let $R$ be a ring and $M$ a (left) $R$-module.
The formal product of a polynomial $f=f(X) = \sum_{i=1}^m a_i X^{\vec{u}_i} \in R[X^{\pm1}]$ and a power series $c=c(X) = \sum_{\vec{u} \in \Z^d} c_{\vec{u}} X^{\vec{u}} \in M[X^{\pm1}]$  is well-defined as the formal power series
$$
fc=f(X) c(X) = \sum_{\vec{u} \in \Z^d} (fc)_{\vec{u}} X^{\vec{u}} \in M[X^{\pm1}]
$$
where
$$
(fc)_{\vec{u}} = \sum_{i=1}^m a_i c_{\vec{u} - \vec{u}_i}.
$$
%is the convolution of $f$ and $c$ (interpreted as functions) at $\vec{u}$.
We say that a polynomial $f=f(X)$ \emph{annihilates} (or \emph{is an annihilator of}) a power series $c=c(X)$ if $fc=0$, that is, if their product is the zero power series.
%, and we say that $c$ is \emph{$\vec{t}$-periodic} if it is annihilated by a \emph{difference polynomial} $X^{\vec{t}}-1$.

In a typical setting, we assume that $\A \subseteq \Z$ and hence consider any configuration $c \in \A^{\Z^d}$ as a finitary and integral power series $c(X)$.
Since multiplying $c(X)$ by the monomial $X^{\vec{u}}$ produces the power series presenting the translation $\tau^{\vec{u}}(c)$ of $c$ by $\vec{u}$,
we have that $c$ is $\vec{u}$-periodic if and only if $c(X)$ is annihilated by the \emph{difference polynomial} $X^{\vec{u}}-1$.
(By a difference polynomial we mean a polynomial $X^{\vec{u}}-1$ for any $\vec{u}\neq 0$.)
This means that it is natural to consider multiplication of $c$ by polynomials in $\C[X^{\pm1}]$.
However, note that the product of $c$ and a polynomial $f \in \C[X^{\pm1}]$ may not be integral, but it is still finitary, hence a configuration.
\iffalse
The annihilator ideal of a power series $c \in \C[[X^{\pm 1}]]$ is the set
$$
\Ann(c) = \{ f \in \C[X^{\pm 1}] \mid fc = 0 \}
$$
which indeed is easily seen to be an ideal of the Laurent polynomial ring $\C[[X^{\pm 1}]]$.
%The elements of the annihilator ideal are called naturally annihilators.
Hence the question whether a configuration (or any formal power series) is periodic is equivalent to asking whether its annihilator ideal contains a difference polynomial.

%If $\Ann(c) \neq \{0\}$ we say that $c$ has a non-trivial annihilator and by a non-trivial annihilator we mean a non-zero annihilator.
Another useful ideal we study is the \emph{periodizer ideal}
$$
\Per(c) = \{ f \in \C[X^{\pm 1}] \mid fc \text{ is strongly periodic} \}.
$$
%Consequently, the elements of the periodizer ideal are called periodizers and any per
Note that clearly $\Ann(c)$ is a subset of $\Per(c)$.
\fi
We say that a polynomial $f$ \emph{periodizes} (or \emph{is a periodizer of}) a configuration $c$ if $fc$ is strongly periodic, that is, periodic in $d$ linearly independent directions.
We denote the set of all periodizers with complex coefficients of a configuration $c$ by $\Per(c)$ which is an ideal of $\C[X^{\pm1}]$ and hence we call it the \emph{periodizer ideal} of $c$.
% More generally, we define the \emph{periodizer ideal}
% of a power series $c(X) \in \C[X^{\pm1}]$ as the set
% $$
% \Per(c) = \{ f \in \C[X^{\pm 1}] \mid fc \text{ is strongly periodic} \}.
% $$
Note that annihilators are periodizers.
Note also that if $c$ has a periodizer $f$, then $(X^{\vec{u}}-1)f$
is an annihilator of $c$ for some $\vec{u}$.
Thus,
$c$ has a non-trivial (= non-zero) annihilator if and only if it has a non-trivial periodizer.
The following theorem states that if a configuration has a non-trivial periodizer, then it has in fact an annihilator of a particular simple form -- a product of difference polynomials.
\iffalse
Finally, recall that for polynomials $f,g$ we say that $g$ is a factor of $f$ if there exists a polynomial $h$ such that $f=gh$.
Consequently, we say that a set of polynomials has a common factor $g$ if $g$ is a factor of every polynomial in the set.
Moreover, we call any non-unit and non-monomial factor non-trivial.
If all the common factors of two polynomials are units we say that they have no common factors.
\fi

\iffalse
Let us briefly review some selected known results on these concepts. First, any low complexity configuration has a non-trivial annihilator:

\begin{lemma}[\cite{icalp}]
    Let $c$ be a low complexity configuration in any dimension. Then it has a non-trivial annihilator.
\end{lemma}
\fi

%\noindent
%If $c$ has a non-trivial annihilator then it has an annihilator of a particular nice form -- a product of difference polynomials:

\begin{theorem}[\cite{icalp}]
    \label{special annihilator}
    Let $c \in \Z[[X^{\pm1}]]$ be a configuration in any dimension and assume that it has a non-trivial periodizer.
    Then there exist $m \geq 1$ and pairwise linearly independent vectors
    $\vec{t}_1, \ldots ,\vec{t}_m$ such that
    $$
    (X^{\vec{t}_1}-1) \cdots (X^{\vec{t}_m}-1)
    $$
    annihilates $c$.
\end{theorem}

\iffalse
\noindent
The previous theorem yields the following decomposition theorem.

\begin{theorem}[Decomposition theorem \cite{fullproofs}]
    Let $c$ be configuration (= a finitary and integral power series) with a non-trivial annihilator. Then there exist periodic integral (but not necessarily finitary) power series $c_1, \ldots ,c_m$ such that
    $$
    c = c_1 + \ldots + c_m.
    $$
\end{theorem}
\fi

%\subsection*{Line polynomials}

%The \emph{support} of a power series $c = \sum_{\vec{u}  \in \Z^d} c_{\vec{u}} X^{\vec{u}}$ is the set $\supp(c) = \{ \vec{u} \in \Z^d \mid c_{\vec{u}} \neq 0 \}$.
%Thus a polynomial is a power series with a finite support.
A \emph{line polynomial} is a polynomial whose support contains at least two points and the points of the support lie on a unique line.
In other words, a polynomial $f$ is a line polynomial if it is not a monomial
and there exist vectors $\vec{u}, \vec{v} \in \Z^d$ such that $\supp(f) \subseteq \vec{u} + \Q \vec{v}$.
In this case we say that $f$ is a line polynomial in direction $\vec{v}$.
We say that non-zero vectors $\vec{v},\vec{v}'\in\Z^d$ are \emph{parallel} if $\vec{v}'\in \Q \vec{v}$, and
clearly then a line polynomial in direction $\vec{v}$ is also a line polynomial in any parallel direction.
A vector $\vec{v}\in\Z^d$ is \emph{primitive} if its components are pairwise relatively prime. If
$\vec{v}$ is primitive, then $\Q \vec{v}\cap\Z^d = \Z \vec{v}$.
For any non-zero $\vec{v}\in\Z^d$ there exists a parallel primitive vector $\vec{v}'\in\Z^d$.
Thus, we may assume the vector $\vec{v}$ in the definition of a line polynomial $f$ to be primitive so that $\supp(f) \subseteq \vec{u} + \Z \vec{v}$.
In the following our preferred presentations of directions are in terms of primitive vectors.

Any line polynomial
$\phi$ in a (primitive) direction $\vec{v}$ can be written uniquely in the form
$$
\phi =
X^{\vec{u}}(a_0 + a_1 X^{\vec{v}} + \ldots + a_n X^{n \vec{v}}) = X^{\vec{u}}(a_0 + a_1 t + \ldots + a_n t^n)
$$
where $\vec{u} \in \Z^d, n \geq 1 , a_0 \neq 0 , a_n \neq 0$ and $t = X^{\vec{v}}$.
Let us call
the single variable proper polynomial $a_0 + a_1 t + \ldots + a_n t^n \in \C[t]$ the \emph{normal form} of $\phi$.
Moreover, for a monomial $a X^{\vec{u}}$ we define its normal form to be $a$.
So, two line polynomials in the direction $\vec{v}$
have the same normal form if and only if
they are the same polynomial up to multiplication by $X^{\vec{u}}$, for some $\vec{u}\in\Z^d$.

Difference polynomials are line polynomials and hence the annihilator provided by Theorem \ref{special annihilator} is a product of line polynomials.
Annihilation by a difference polynomial means periodicity.
More generally,
annihilation of a configuration $c$ by a line polynomial in a primitive direction $\vec{v}$
can be understood as the annihilation of the one-dimensional \emph{$\vec{v}$-fibers} $\sum_{k \in \Z} c_{\vec{u} + k \vec{v}} X^{\vec{u}+k\vec{v}}$ of $c$ in
direction $\vec{v}$, and since annihilation in the one-dimensional setting implies periodicity with a bounded period, we conclude that a configuration is periodic if and only if it is annihilated by a line polynomial.
It is known that if $c$ has a periodizer with line polynomial factors in at most one primitive direction, then $c$ is periodic:

\begin{theorem}[\cite{fullproofs}] \label{theorem on line polynomial factors}
    Let $c \in \Z[[x^{\pm1},y^{\pm1}]]$ be a two-dimensional configuration and let $f$ be a periodizer of $c$. Then the following conditions hold.
    \begin{itemize}
        \item If $f$ does not have any line polynomial factors, then $c$ is two-periodic.
        \item If all line polynomial factors of $f$ are in the same primitive direction, then $c$ is periodic in this direction.
    \end{itemize}
\end{theorem}

\noindent
\emph{Proof sketch.}
The periodizer ideal $\Per(c) = \{ g \in \C[x^{\pm1}, y^{\pm1}] \mid gc \text{ is two-periodic} \}$
of $c$
is a principal ideal generated by a polynomial $g = \phi_1 \cdots \phi_m$  where $\phi_1, \ldots ,\phi_m$ are line polynomials in pairwise non-parallel directions~\cite{fullproofs}.
Because $f\in\Per(c)$, we know that $g$ divides $f$.
If $f$ does not have any line polynomial factors, then $g = 1$ and hence $c = gc$ is two-periodic.
If $f$ has line polynomial factors, and they are in the same primitive
direction $\vec{v}$, then $g$ is a line polynomial in this direction.
Since $gc$ is two-periodic, it is annihilated by $(X^{k \vec{v}} - 1)$ for some $k \in \Z$.
This implies that the configuration $c$ is annihilated by the line polynomial $(X^{k \vec{v}} - 1)g$
in direction $\vec{v}$.
We conclude that $c$ is periodic in direction $\vec{v}$.
\qed

\medskip

%\noindent
%(See the Appendix for an alternative proof that mimics the usage of resultants in~\cite{karimoutot},
%instead of relying on the structure of the ideal $\Per(c)$.)

\noindent
The proof of the previous theorem sketched above relies heavily on the structure of the ideal $\Per(c)$ developed in \cite{icalp}.
We give an alternative proof sketch that mimics the usage of resultants in \cite{karimoutot}:

\medskip
\noindent
\emph{Second proof sketch of Theorem \ref{theorem on line polynomial factors}.}
The existence of a non-trivial periodizer $f$ implies
by Theorem~\ref{special annihilator} that $c$ has a special annihilator
$g=\phi_1 \cdots \phi_m$ that is a product of (difference) line polynomials $\phi_1, \ldots ,\phi_m$ in pairwise non-parallel
directions.
All irreducible factors of $g$ are line polynomials. If $f$ does not have any line polynomial factors, then
the periodizers $f$ and $g$ do not have common factors.
We can assume that both are proper polynomials as
they can be multiplied by a suitable monomial if needed.
% The \emph{$x$-resultant} of $f,g\in \C[x,y]$
% is a polynomial $\Res_x(f,g)=\alpha f+\beta g$ for some $\alpha,\beta\in \C[x,y]$ such that the variable $x$ is
% eliminated, {\it i.e.},
% $\Res_x(f,g)$ is a polynomial in variable $y$ only. Moreover, since $f$ and $g$ do not have common factors,
% $\Res_x(f,g) \neq 0$.
Because $f,g\in\Per(c)$, also their resultant $\Res_x(f,g)\in \Per(c)$, implying that
$c$ has a non-trivial annihilator containing only variable $y$ since $\Res_x(f,g) \neq 0$ because $f$ and $g$ have no common factors.
This means that $c$ is periodic in the vertical direction.
Analogously, the \emph{$y$-resultant} $\Res_y(f,g)$ shows that $c$ is horizontally periodic, and hence two-periodic.

The proof for the case that $f$ has line polynomial factors only in one direction $\vec{v}$
goes analogously by considering $\phi c$
instead of $c$, where $\phi$ is the greatest common line polynomial factor of $f$ and $g$ in the direction $\vec{v}$.
We get that $\phi c$ is two-periodic, implying that $c$ is periodic in direction $\vec{v}$.
\qed
\bigskip
%\subsection*{Power series with vector coefficients}

%In this paper we also consider the sets $\C^n[[X^{\pm1}]]$ and $\C^{n\times n}[X^{\pm1}]$ of formal power series with complex vector coefficients and (Laurent) polynomials with complex matrix coefficients, respectively.
%In fact we deal mostly with integer vectors and matrices, that is, we consider
%the sets $\Z^n[[X^{\pm1}]]$ and $\Z^{n\times n}[X^{\pm1}]$.
In this paper we also consider configurations over alphabets $\A$ that are finite subsets of $\Z^{n}$, that is, the set of length $n$ integer vectors, and hence study finitary formal power series from the set
$\Z^n[[X^{\pm1}]]$ for $n \geq 2$.
In particular, we call this kind of configurations \emph{integral vector configurations}.
Also in this setting we consider multiplication of power series by polynomials.
The coefficients of the polynomials are $n \times n$ integer matrices, {\it i.e.}, elements of the ring $\Z^{n \times n}$.
Since $\Z^n$ is a (left) $\Z^{n \times n}$-module where we consider the vectors of $\Z^n$ as column vectors, the product of a polynomial $f= f(X) \in \Z^{n\times n}[X^{\pm1}]$ and a power series $c = c(X) \in \Z^n[[X^{\pm1}]]$
is well-defined.
Consequently, we say that $c(X) \in \Z^n[[X^{\pm1}]]$ is $\vec{t}$-periodic if it is annihilated by the polynomial $\vec{I} X^{\vec{t}} - \vec{I}$ and that it is periodic if it is $\vec{t}$-periodic for some non-zero $\vec{t}$.
%Note also that the case $n = 1$ corresponds to the
%Note also that in place of $\Z$ we could consider any ring $R$ since the set $R^n$ is an $R^{n\times n}$-module.
%Naturally, we say that $f \in \Z^{n\times n}[X^{\pm1}]$ annihilates $c \in \C^n[[X^{\pm1}]]$ if
%$fc = \sum_{\vec{u} \in \Z^d} \vec{0} X^{\vec{u}} = 0$.
%Consequently,
%$c \in \Z^n[[X^{\pm1}]]$ is $\vec{t}$-periodic if it is annihilated by $\mathbf{I} X^{\vec{t}} - \mathbf{I}$ and periodic if it is $\vec{t}$-periodic for some non-zero $\vec{t}$ where $\mathbf{I}$ is the $n \times n$ identity matrix.

There is a natural way to present configurations over arbitrary alphabets as integral vector configurations.
Let $\A = \{ a_1, \ldots , a_n \}$ be a finite alphabet with $n$ elements.
The \emph{vector presentation} of a configuration $c \in \A^{\Z^d}$ is the configuration $c' \in \{ \vec{e}_1, \ldots , \vec{e}_n \} ^{\Z^d}$ (or the power series $c'(X) \in \Z^n[[X^{\pm1}]]$ presenting $c'$) defined such that $c'_{\vec{u}} = \vec{e}_i$ if and only if $c_{\vec{u}} = a_i$.
Here by $\vec{e}_i \in \Z^n$ we denote the $i$th natural base vector, {\it i.e.}, the vector whose $i$th component is 1 while all the other components are 0.
Clearly $c$ is $\vec{t}$-periodic if and only if its vector presentation is $\vec{t}$-periodic.
Thus, to study the periodicity of a configuration it is sufficient to study the periodicity of its vector presentation.

The $i$th \emph{layer} of $c = \sum \vec{c}_{\vec{u}} X^{\vec{u}} \in \Z^n[[X^{\pm1}]]$ is the power series
$$
\text{layer}_i(c) = \sum c_{\vec{u}}^{(i)} X^{\vec{u}} \in \Z[[X^{\pm1}]]
$$
where $c_{\vec{u}}^{(i)}$ is the $i$th component of $\vec{c}_{\vec{u}}$.
Clearly $c \in \Z^n[[X^{\pm1}]]$ is periodic in direction $\vec{v}$ if and only if for all $i \in \{1,\ldots,n\}$ the $i$th layer of $c$ is periodic in direction $\vec{v}$.

%Finally, let us illustrate how to present additive CA in terms of annihilating polynomials.
Finally, let $R$ be a finite ring and $\A$ a finite $R$-module.
A polynomial $f(X) = \sum_{i=1}^n a_i X^{-\vec{u}_i} \in R[x_1^{\pm 1}, \ldots , x_d^{\pm1}]$
defines an additive CA that
has neighborhood vector $(\vec{u}_1,\ldots,\vec{u}_n)$ and local rule
$f'(y_1,\ldots,y_n) = a_1y_1 + \ldots + a_ny_n$.
% Examples: $R=M=\Z_p$ or $R=\Z_P^{n \times n}, M=\Z_p^n$......
More precisely, the image of a configuration $c$ under the CA determined by $f$ is the configuration $fc$.

%\newpage

\section{Finding the line polynomial factors of a given two-variate Laurent polynomial} \label{line polynomial factors}

In this section we have $d=2$ and hence all our polynomials are in two variables $x$ and $y$.
The open and closed \emph{discrete half planes} determined by a non-zero vector $\vec{v} \in \Z^2$ are the sets $H_{\vec{v}} = \{ \vec{u} \in \Z^2 \mid \langle \vec{u} , \vec{v} ^{\perp} \rangle > 0 \}$ and $\overline{H}_{\vec{v}} = \{ \vec{u} \in \Z^2 \mid \langle \vec{u} , \vec{v} ^{\perp} \rangle \geq 0 \}$, respectively, where $\vec{v}^{\perp} = (v_2,-v_1)$ is orthogonal to $\vec{v} = (v_1,v_2)$.
Let us also denote by $l_{\vec{v}} = \overline{H}_{\vec{v}} \setminus H_{\vec{v}}$ the discrete line parallel to $\vec{v}$ that goes through the origin.
In other words, the half plane determined by $\vec{v}$ is the half plane ``to the right'' of the line $l_{\vec{v}}$ when moving along the line in the direction of $\vec{v}$.
We say that a finite set $D \subseteq \Z^2$ has an \emph{outer edge} in direction $\vec{v}$ if there exists a vector $\vec{t} \in \Z^2$ such that $D \subseteq \overline{H}_{\vec{v}} + \vec{t}$ and $|D \cap (l_{\vec{v}} + \vec{t})| \geq 2$.
We call $D \cap (l_{\vec{v}} + \vec{t})$ the outer edge of $D$ in direction $\vec{v}$.
An outer edge corresponding to $\vec{v}$ means that the convex hull of $D$  has an edge in direction $\vec{v}$
in the clockwise orientation around $D$.

If a finite non-empty set $D$ does not have an outer edge in direction $\vec{v}$, then there exists a vector $\vec{t} \in \Z^2$ such that $D \subseteq \overline{H}_{\vec{v}} + \vec{t}$ and $|D \cap (l_{\vec{v}} + \vec{t})| = 1$, and then we say that $D$ has a vertex in direction $\vec{v}$. We call $D \cap (l_{\vec{v}} + \vec{t})$ the vertex of $D$ in direction $\vec{v}$.
We say that a polynomial $f$ has an outer edge or a vertex in direction $\vec{v}$ if its support has an outer edge or a vertex in direction $\vec{v}$, respectively.
%The outer edges of a finite shape $D$ are the edges of its convex hull and the vertices of $D$ are the vertices of its convex hull.
Note that every non-empty finite shape $D$ has either an edge or a vertex in any non-zero direction. Note also that in this context
directions $\vec{v}$ and $-\vec{v}$ are not the same: a shape may have an outer edge in direction $\vec{v}$
but no outer edge in direction $-\vec{v}$.
The following lemma shows that a polynomial can have line polynomial factors only in the directions of its outer edges.

\begin{lemma}[\cite{karimoutot}]
\label{lemma1}
    Let $f$ be a non-zero polynomial with a line polynomial factor in direction $\vec{v}$. Then $f$ has outer edges in directions $\vec{v}$ and $-\vec{v}$.
\end{lemma}

%\noindent
Let $\vec{v} \in \Z^2 \setminus \{ \vec{0} \}$ be a non-zero primitive vector
and let $f = \sum f_{\vec{u}} X^{\vec{u}}$ be a polynomial.
Recall that a \emph{$\vec{v}$-fiber} of $f$ is a polynomial
of the form $$\sum_{k \in \Z} f_{\vec{u} + k \vec{v}} X^{\vec{u} + k \vec{v}}$$ for some $\vec{u}\in\Z^2$.
Thus, a non-zero $\vec{v}$-fiber of a polynomial is either a line polynomial or a monomial.
%Clearly any polynomial has only finitely many non-trivial (= non-zero) $\vec{v}$-fibers.
Let us denote by $\mathcal{F}_{\vec{v}}(f)$ the set of different normal forms of all non-zero
$\vec{v}$-fibers of a polynomial $f$, which is hence a finite set of one-variate proper polynomials.
The following simple example illustrates the concept of fibers and their normal forms.

\begin{example}
    Let us determine the set $\mathcal{F}_{\vec{v}}(f)$ for $f = f(X) = f(x,y) = 3x + y + xy^2 +xy + x^3y^3 + x^4y^4$
    and $\vec{v} = (1,1)$.
    By grouping the terms we can write
    $$
    f = 3x + y(1 + xy) + xy(1 + x^2y^2 + x^3y^3)
    = X^{(1,0)} \cdot 3 + X^{(0,1)}(1 + t) + X^{(1,1)}(1 + t^2 + t^3)
    $$
    where $t = X^{(1,1)} = xy$.
    Hence,
    $\mathcal{F}_{\vec{v}}(f) =
    \{ 3, 1 + t, 1 + t^2 + t^3 \}$.
    See Figure \ref{havainnollistus} for a pictorial illustration.
    \qed
\end{example}

\begin{figure}[ht]
	\centering
	\begin{tikzpicture}[scale=0.6]
	    \draw (-0.5,-0.5) grid (4.5,4.5);
	    \draw[very thick, ->] (0,-0.5) -- (0,4.5);
	    \draw[very thick, ->] (-0.5,0) -- (4.5,0);
	
		\draw[red, fill=red] (1,0) circle(3pt);
		\node at (1,0.3) {\tiny $3x$};
		
		\draw[blue, fill=blue] (0,1) circle(3pt);
		\node at (0.1,1.3) {\tiny $y$};
		
		\draw[blue,fill=blue] (1,2) circle(3pt);
		\node at (1,2.35) {\tiny $xy^2$};
		
		\draw[green,fill=green] (1,1) circle(3pt);
		\node at (0.9,1.4) {\tiny $xy$};
		
		\draw[green,fill=green] (3,3) circle(3pt);
		\node at (2.9,3.4) {\tiny $x^3y^3$};
		
		\draw[green,fill=green] (4,4) circle(3pt);
		\node at (3.9,4.4) {\tiny $x^4y^4$};
		
		\draw[blue] (-0.5,0.5) -- (1.5,2.5);
		
		\draw[green] (0.5,0.5) -- (4.5,4.5);
		
		\draw[red] (0.5,-0.5) -- (1.5,0.5);
	\end{tikzpicture}
	\caption{The support of $f = 3x + y + xy^2 +xy + x^3y^3 + x^4y^4$ and its different $(1,1)$-fibers.}
	\label{havainnollistus}
\end{figure}
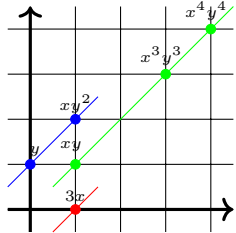

\noindent
As noticed in the example above, polynomials are linear combinations of their fibers: for any polynomial $f$ and any non-zero primitive vector $\vec{v}$ we can write
$$
f = X^{\vec{u}_1} \psi_1 + \ldots + X^{\vec{u}_n} \psi_n
$$
for some $\vec{u}_1, \ldots , \vec{u}_n \in \Z^2$ where $\psi_1, \ldots , \psi_n\in \mathcal{F}_{\vec{v}}(f)$.
We use this in the proof of the next theorem.

\begin{theorem} \label{theorem1}
	A polynomial $f$ has a line polynomial factor in direction $\vec{v}$ if and only if the polynomials in $\mathcal{F}_{\vec{v}}(f)$ have a common factor.
\end{theorem}

\begin{proof}
For any line polynomial $\phi$ in direction $\vec{v}$, and for any polynomial $g$, the $\vec{v}$-fibers of the
product $\phi g$ have a common factor $\phi$.
In other words, if a polynomial $f$ has a line polynomial factor $\phi$ in direction $\vec{v}$, then the polynomials in $\mathcal{F}_{\vec{v}}(f)$ have the normal form of $\phi$ as a common factor.

For the converse direction, assume that
the polynomials in $\mathcal{F}_{\vec{v}}(f)$ have a common factor $\phi$.
Then there exist vectors $\vec{u}_1, \ldots , \vec{u}_n \in \Z^2$ and polynomials $\phi\psi_1, \ldots , \phi\psi_n \in\mathcal{F}_{\vec{v}}(f)$ such that
$$
f = X^{\vec{u}_1} \phi \psi_1 + \ldots + X^{\vec{u}_n} \phi \psi_n.
$$
Hence,
$\phi$ is a line polynomial factor of $f$ in direction $\vec{v}$.
\end{proof}

\noindent
Note that Lemma~\ref{lemma1} actually follows immediately from Theorem~\ref{theorem1}: A vertex instead of an outer edge
in direction $\vec{v}$ or $-\vec{v}$ provides a non-zero monomial $\vec{v}$-fiber, which implies that
the polynomials in $\mathcal{F}_{\vec{v}}(f)$ have no common factors.

\iffalse
\noindent
The previous theorem and Theorem \ref{theorem on line polynomial factors}
yield the following result.

\begin{corollary}
    Let $c$ be a two-dimensional configuration and $f \in \Per(c)$. Then the following holds.
    \begin{itemize}
        \item If for all $\vec{v}$ the polynomials in $\mathcal{F}_{\vec{v}}(f)$ have no
        %non-trivial
        common factors for then $c$ is two-periodic.
        \item If there exist a unique direction $\langle \vec{v} \rangle$ such that the polynomials in $\mathcal{F}_{\vec{v}}(f)$ have a
        %non-trivial
        common factor then $c$ is periodic in direction $\vec{v}$.
    \end{itemize}
\end{corollary}
\fi

%\noindent
So, to find out the line polynomial factors of $f$ we first
need to find out the possible directions of the line polynomials, that is, the directions of the (finitely many)
outer edges of $f$, and then we need to check for which of these possible directions $\vec{v}$ the polynomials
in $\mathcal{F}_{\vec{v}}(f)$ have a common factor.
There are clearly algorithms to find the outer edges of a given polynomial
and to determine whether finitely many line polynomials have a common factor.
If such a factor exists, then by Theorem \ref{theorem1} the polynomial $f$ has a line polynomial factor in this direction.
We have proved the following theorem.

\begin{theorem}
    There is an algorithm to find the line polynomial factors of a given (Laurent) polynomial in two variables.
\end{theorem}

%\newpage

\section{Forced periodicity of perfect colorings with two colors}
\label{perfect coverings}

%\subsection*{Infinite grid graphs}

In this section we consider forced periodicity of two-dimensional perfect colorings with only two colors.
Without loss of generality we may assume that $\A = \{a_1, a_2 \} = \{ 0 , 1 \}$ ($a_1 =0,a_2=1$) and consider perfect colorings $c \in \A^{\Z^2}$ since the names of the colors do not matter in our considerations.
So, let $c \in \{0,1\}^{\Z^2}$ be a perfect coloring with respect to $D \subseteq \Z^2$ and let $\mathbf{B} = (b_{ij})_{2 \times 2}$ be the matrix of $c$.
Let us define
a set
$C = \{ \vec{u} \in \Z^2 \mid c_{\vec{u}} = 1 \}$.
This set has the property that the neighborhood $\vec{u} + D$ of a point $\vec{u}$ contains exactly $a = b_{21}$ points of color $1$ if $\vec{u} \not \in C$ and exactly $b = b_{22}$ points of color $1$ if $\vec{u} \in C$.
In fact, $C$ is a \emph{perfect (multiple) covering} of the infinite grid $G$ determined by the relative neighborhood $D$.
More precisely, the set $C$ is a (perfect) \emph{$(D,b,a)$-covering} of $G$.
This is a variant of the following definition: in any graph a subset $C$ of its vertex set is an \emph{$(r,b,a)$-covering} if the number of vertices of $C$ in the $r$-neighborhood of a vertex $u$ is $a$ if $u \not \in C$ and $b$ if $u \in C$.
See \cite{Axenovich} for a reference.
Clearly in translation invariant graphs the $(r,b,a)$-coverings correspond to $(D,b,a)$-coverings where $D$ is the relative $r$-neighborhood of the graph.
Thus, it is natural to call any perfect coloring with only two colors a perfect covering.
Note that a $(D,b,a)$-covering is a $D$-perfect coloring with the matrix
$$
\mathbf{B} =
\begin{pmatrix}
|D| - a & |D| - b \\
a & b
\end{pmatrix}.
$$
%Thus for perfect colorings with two colors it is sufficient to give only the numbers
%\cite{Axenovich,godsil,puzynina}

%We consider the forced periodicity of perfect coverings in the square grid, the triangular grid and the king grid.
%We give a simple algebraic proof for the following result by Axenovich.
The following theorem by Axenovich states that ``almost every'' $(1,b,a)$-covering in the square grid is two-periodic.

\begin{theorem}[\cite{Axenovich}] \label{squaregrid1}
    If $b-a \neq 1$, then every $(1,b,a)$-covering in the square grid is two-periodic.
\end{theorem}

\noindent
%Before reproving the above theorem let us show how we can treat perfect coverings algebraically.
For a finite set $D \subseteq \Z^2$ we define its \emph{characteristic polynomial} to be the polynomial $f_D(X) = \sum _{\vec{u} \in D} X^{-\vec{u}}$.
We denote by $\mathbbm{1}(X)$
the constant power series $\sum_{\vec{u} \in \Z^2} X^{\vec{u}}$.
If $c \in \{0,1\}^{\Z^2}$ is a $(D,b,a)$-covering, then from the definition we get
that
$
f_D(X)c(X) = (b-a)c(X) + a \mathbbm{1}(X)
$
which is equivalent
to
$\left (f_D(X) - (b-a) \right )c(X) = a \mathbbm{1}(X)$.
Thus, if $c$ is a $(D,b,a)$-covering, then $f_D(X) - (b-a)$ is a periodizer of $c$.
Hence, by Theorem \ref{theorem on line polynomial factors} the condition that the polynomial
$f_D(X) - (b-a)$
has no line polynomial factors is a sufficient condition for forced periodicity of a $(D,b,a)$-covering.
%This is something we will use repeatedly in the following considerations.
%More precisely, we will study whether the polynomial $f_D - (b-a) \in \Per(c)$ has any line polynomial factors and use Theorem \ref{theorem on line polynomial factors} if it has no line polynomial factors to conclude that in these cases $c$ is in fact two-periodic.
% Using our formulation and the algebraic approach we get a simple proof for Theorem \ref{squaregrid1}.
Hence, we have the following corollary of Theorem \ref{theorem on line polynomial factors}:

\begin{corollary} \label{corollary1}
    Let $D \subseteq \Z^2$ be a finite shape and let $b$ and $b$ be non-negative integers.
    If $g=f_D - (b-a)$ has no line polynomial factors, then every $(D,b,a)$-covering is two-periodic.
\end{corollary}

\noindent
Using our formulation and the algebraic approach we get a simple proof for Theorem \ref{squaregrid1}:

\setcounter{reformulation}{6}
\begin{reformulation}
    Let $D$ be the relative 1-neighborhood of the square grid and assume that $b-a \neq 1$.
    Then every $(D,b,a)$-covering is two-periodic.
\end{reformulation}

\begin{proof}
    %Let $D = \{ (-1,0),(0,-1) ,(0,0),(1,0),(0,1) \}$ be the 1-neighborhood of $\vec{0}$ in the square grid and assume that $b-a \neq 1$.
	Let $c$ be an arbitrary $(D,b,a)$-covering.
    % We show that $g = f_D - (b-a) = x^{-1} + y^{-1} + 1 - (b-a) + x + y  \in \Per(c)$ has no line polynomial factors. By Theorem \ref{theorem on line polynomial factors} the configuration $c$ is two-periodic.
    The outer edges of $g = f_D - (b-a) = x^{-1} + y^{-1} + 1 - (b-a) + x + y$ are in directions $(1,1),(-1,-1),(1,-1)$ and $(-1,1)$ and hence
    by Lemma~\ref{lemma1}
    any line polynomial factor of $g$ is either in direction $(1,1)$ or $(1,-1)$.
    For $\vec{v} \in \{(1,1),(1,-1)\}$
    we have $\mathcal{F}_{\vec{v}}(g) = \{ 1+t,1 -(b-a) \}$.
    See Figure \ref{Illustrations} for an illustration.
    Since $1 -(b-a)$ is a non-trivial monomial,
    by Theorem \ref{theorem1} the periodizer $g \in \Per(c)$ has no line polynomial factors and hence the claim follows by corollary \ref{corollary1}.
    % in any direction.
    % and hence $c$ is two-periodic.
    %\qed
\end{proof}

\noindent
We also get a similar proof for the following known result concerning the forced periodicity perfect coverings in the square grid with radius $r \geq 2$.

\begin{theorem}[\cite{puzynina2}] \label{squaregrid2}
    Let $r \geq 2$ and let $D$ be the relative $r$-neighborhood of the square grid.
    Then every $(D,b,a)$-covering is two-periodic.
    In other words, all $(r,b,a)$-coverings in the square grid are two-periodic for all $r \geq 2$.
\end{theorem}

\begin{proof}
    Let $c$ be an arbitrary $(D,b,a)$-covering.
    % Again, by Theorem \ref{theorem on line polynomial factors}, it is enough to
    % show that $g = f_D -(b-a) \in \Per(c)$ has no line polynomial factors.
    By Lemma~\ref{lemma1} any line polynomial factor of $g = f_D -(b-a)$
    has direction
    $(1,1)$ or $(1,-1)$.
    So, assume that $\vec{v} \in \{ (1,1), (1,-1) \}$.
    We have $\phi_1=1+t+\ldots+t^r \in \mathcal{F}_{\vec{v}}(g)$ and
    $\phi_2=1+t+\ldots+t^{r-1} \in \mathcal{F}_{\vec{v}}(g)$.
    See Figure \ref{Illustrations} for an illustration in the case $r=2$.
    Since $\phi_1-\phi_2=t^r$, the polynomials
    $\phi_1$ and $\phi_2$ have no common factors, and hence by Theorem \ref{theorem1} the periodizer
    $g$ has no line polynomial factors.
    Corollary \ref{corollary1} gives the claim.
    %Thus
    %$c$ is two-periodic.
\end{proof}

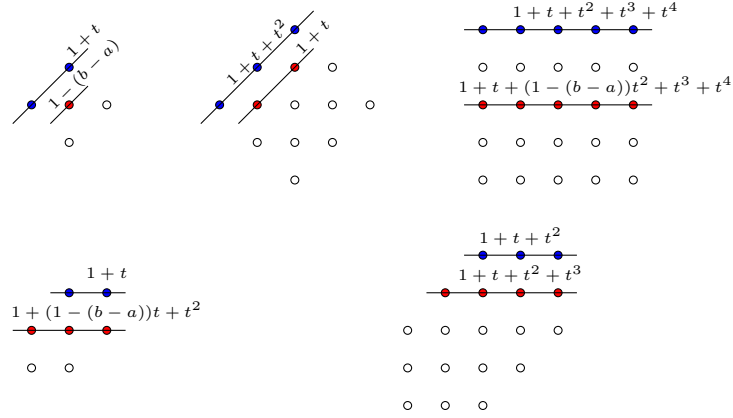
\begin{figure}
	\centering
	\begin{tikzpicture}[scale=0.5]
		% square grid neighborhood r=1
		\draw[fill=red] (0,0) circle(3pt);
		\draw[fill=blue] (-1,0) circle(3pt);
		\draw[] (0,-1) circle(3pt);
		\draw[] (1,0) circle(3pt);
		\draw[fill=blue] (0,1) circle(3pt);
		
		\draw (-1.5,-0.5) -- (0.5,1.5);
		%\draw (-0.5,-1.5) -- (1.5,0.5);
		\draw (-0.5,-0.5) -- (0.5,0.5);
		\node[rotate=45] at (0.4,1.7) {\tiny $1+t$};	
		%\node[rotate=45] at (1.4,0.7) {\tiny $1+t$};
		\node[rotate=45] at (0.4,0.8) {\tiny $1-(b-a)$};	
		
		%\node at (0,-2) {(a) The set $\mathcal{F}_{\vec{v}}(g)$};
		
		%\draw[densely dotted] (0,1.4) -- (1.4,0) -- (0,-1.4) -- (-1.4,0) -- (0,1.4);

	%square grid r=2
	
		\draw[] (6,0) circle(3pt);
		\draw[fill=red] (5,0) circle(3pt);
		\draw[] (6,-1) circle(3pt);
		\draw[] (7,0) circle(3pt);
		\draw[fill=red] (6,1) circle(3pt);
		\draw[fill=blue] (6,2) circle(3pt);
		\draw[] (6,-2) circle(3pt);
		\draw[] (7,1) circle(3pt);
		\draw[] (7,-1) circle(3pt);
		\draw[fill=blue] (5,1) circle(3pt);
		\draw[] (5,-1) circle(3pt);
		\draw[fill=blue] (4,0) circle(3pt);
		\draw[] (8,0) circle(3pt);
		
		\draw (3.5,-0.5) -- (6.5,2.5);
		\draw (4.5,-0.5) -- (6.5,1.5);
		
		\node[rotate=45] at (5,1.4) {\tiny $1+t+t^2$};
		\node[rotate=45] at (6.5,1.8) {\tiny $1+t$};

		%\draw[densely dotted] (6,2.5) -- (8.5,0) -- (6,-2.5) -- (3.5,0) -- (6,2.5);
		
		%king grid r=2
		
		\draw[fill=red] (13,0) circle(3pt);
		\draw[fill=red] (12,0) circle(3pt);
		\draw[] (13,-1) circle(3pt);
		\draw[fill=red] (14,0) circle(3pt);
		\draw[] (13,1) circle(3pt);
		\draw[fill=blue] (13,2) circle(3pt);
		\draw[] (13,-2) circle(3pt);
		\draw[] (14,1) circle(3pt);
		\draw[] (14,-1) circle(3pt);
		\draw[] (12,1) circle(3pt);
		\draw[] (12,-1) circle(3pt);
		\draw[fill=red] (11,0) circle(3pt);
		\draw[fill=red] (15,0) circle(3pt);
		\draw[] (15,1) circle(3pt);
		\draw[fill=blue] (15,2) circle(3pt);
		\draw[] (15,-1) circle(3pt);
		\draw[] (15,-2) circle(3pt);
		\draw[] (11,1) circle(3pt);
		\draw[fill=blue] (11,2) circle(3pt);
		\draw[] (11,-1) circle(3pt);
		\draw[] (11,-2) circle(3pt);
		\draw[fill=blue] (12,2) circle(3pt);
		\draw[] (12,-2) circle(3pt);
		\draw[fill=blue] (14,2) circle(3pt);
		\draw[] (14,-2) circle(3pt);
		
		\draw (10.5,2) -- (15.5,2);
		\draw (10.5,0) -- (15.5,0);
		\node at (14,2.5) {\tiny $1+t+t^2+t^3+t^4$};
		\node at (14,0.5) {\tiny $1+t+(1-(b-a))t^2+t^3+t^4$};

		% triangular grid neighborhood r=1
		\draw[fill=red] (0,-6) circle(3pt);
		\draw[fill=red] (-1,-6) circle(3pt);
		\draw[] (0,-7) circle(3pt);
		\draw[fill=red] (1,-6) circle(3pt);
		\draw[fill=blue] (0,-5) circle(3pt);
		\draw[fill=blue] (1,-5) circle(3pt);
		\draw[] (-1,-7) circle(3pt);
	
		\draw (-0.5,-5) -- (1.5,-5);
		\draw (-1.5,-6) -- (1.5,-6);
		\node at (1,-4.5) {\tiny $1+t$};
		\node at (1,-5.5) {\tiny $1+(1-(b-a))t+t^2$};
	
	%triangular grid r=2
	
		\draw[] (11,-6) circle(3pt);
		\draw[] (10,-6) circle(3pt);
		\draw[] (11,-7) circle(3pt);
		\draw[] (12,-6) circle(3pt);
		\draw[fill=red] (11,-5) circle(3pt);
		\draw[fill=blue] (11,-4) circle(3pt);
		\draw[] (11,-8) circle(3pt);
		\draw[fill=red] (12,-5) circle(3pt);
		\draw[] (12,-7) circle(3pt);
		\draw[fill=red] (10,-5) circle(3pt);
		\draw[] (10,-7) circle(3pt);
		\draw[] (9,-6) circle(3pt);
		\draw[] (13,-6) circle(3pt);
		\draw[] (9,-7) circle(3pt);
		\draw[] (9,-8) circle(3pt);
		\draw[] (10,-8) circle(3pt);
		\draw[fill=blue] (12,-4) circle(3pt);
		\draw[fill=blue] (13,-4) circle(3pt);
		\draw[fill=red] (13,-5) circle(3pt);
		
		\draw (10.5,-4) -- (13.5,-4);
		\draw (9.5,-5) -- (13.5,-5);
		
		\node at (12,-3.5) {\tiny $1+t+t^2$};
		\node at (12,-4.5) {\tiny $1+t+t^2+t^3$};
	\end{tikzpicture}
	\caption{Pictorial illustrations for the proofs of Theorems \ref{squaregrid1}, \ref{squaregrid2}, \ref{triangulargrid1}, \ref{triangulargrid2} and \ref{kinggrid}.
	The constellation on the left of the upper row illustrates
	%how to read the different normal forms in $\mathcal{F}_{\vec{v}}(g)$ for $\vec{v} = (1,1)$ in
	the proof of Theorem \ref{squaregrid1}.
	The constellation in the center of the upper row illustrates the proof of Theorem \ref{squaregrid2} with $r=2$.
	The constellation on the right of the upper row illustrates the proof of Theorem \ref{kinggrid} with $r=2$.
	The constellation on the left of the lower row illustrates the proof of Theorem \ref{triangulargrid1}.
	The constellation on the right of the lower row illustrates the proof of Theorem \ref{triangulargrid2} with $r=2$.
	In each of the constellations we have pointed out two normal forms with no common factors in $\mathcal{F}_{\vec{v}}(g)$ from the points of $\supp(g)$ for one of the outer edges $\vec{v}$ of $\supp(g)$.
	}
    \label{Illustrations}
\end{figure}

\noindent
There are analogous results in the triangular grid, and we can prove them similarly using Corollary \ref{corollary1}.

\begin{theorem}[\cite{puzynina2}] \label{triangulargrid1}
    Let $D$ be the relative 1-neighborhood of the triangular grid and assume that
    $b - a \neq -1$.
    Then every $(D,b,a)$-covering in the triangular grid is two-periodic.
    In other words, all $(1,b,a)$-coverings in the triangular grid are two-periodic whenever $b-a \neq -1$.
\end{theorem}

%\iffalse
\begin{proof}
    %Let $D = \{ (-1,-1),(-1,0),(0,-1),(0,0),(1,0),(0,1),(1,1) \}$ be the 1-neighborhood of $\vec{0}$ in the triangular grid and assume that
    %$b-a \neq -1$.
    Let $c$ be an arbitrary $(D,b,a)$-covering.
    % Once again, we show that $g = f_D -(b-a) = x^{-1}y^{-1} + x^{-1} + y^{-1} + 1-(b-a) + x + y + xy$ has no line polynomial factors, so that by Theorem \ref{theorem on line polynomial factors} the configuration $c$ is two-periodic.
    The outer edges of $g = f_D -(b-a) = x^{-1}y^{-1} + x^{-1} + y^{-1} + 1-(b-a) + x + y + xy$ have directions
    $(1,1),(-1,-1),(1,0),(-1,0)$, $(0,1)$ and $(0,-1)$
    and hence by Lemma~\ref{lemma1} any line polynomial factor of $g$ has direction $(1,1)$, $(1,0)$ or $(0,1)$.
    So, let $\vec{v} \in \{ (1,1), (1,0),(0,1) \}$.
    We have
    $\mathcal{F}_{\vec{v}}(g) = \{ 1+t, 1+ (1 - (b-a))t + t^2 \}$.
    See Figure \ref{Illustrations} for an illustration.
    Polynomials $\phi_1=1+t$ and $\phi_2=1+ (1 - (b-a))t + t^2$ satisfy $\phi_1^2-\phi_2= (1+b-a)t$.
    Thus,
    they do not have any common factors if $b-a \neq -1$ and hence by Theorem \ref{theorem1} the polynomial
    $g$ has no line polynomial factors.
    The claim follows by Corollary \ref{corollary1}.
    % and hence $c$ is two-periodic.
    %\qed
\end{proof}

\begin{theorem}[\cite{puzynina2}] \label{triangulargrid2}
    Let $r \geq 2$ and let $D$ be the relative $r$-neighborhood of the triangular grid.
    Then every $(D,b,a)$-covering is two-periodic.
    In other words, every $(r,b,a)$-covering in the triangular grid is two-periodic for all $r \geq 2$.
\end{theorem}

\begin{proof}
    Let $c$ be an arbitrary $(D,b,a)$-covering.
    % We show that $g = f_D -(b-a) \in \Per(c)$ has no line polynomial factors, which by Theorem \ref{theorem on line polynomial factors} implies that the configuration $c$ is two-periodic.
    The outer edges
    of $g = f_D -(b-a)$ have directions
    $(1,1)$, $(-1,-1)$, $(1,0)$, $(-1,0)$, $(0,1)$ and $(0,-1)$,
    and hence by Lemma~\ref{lemma1} any line polynomial factor of $g$
    has direction $(1,1)$, $(1,0)$ or $(0,1)$.
    So, let $\vec{v} \in \{ (1,1), (1,0),(0,1) \}$.
    There exists $n \geq 1$ such that $1+t+ \ldots + t^n \in \mathcal{F}_{\vec{v}}(g)$ and $1+t+ \ldots + t^{n+1} \in \mathcal{F}_{\vec{v}}(g)$.
    See Figure \ref{Illustrations} for an illustration with $r=2$.
    Since these two polynomials have no common factors, by Theorem \ref{theorem1} the polynomial
    $g$ has no line polynomial factors.
    Again, Corollary \ref{corollary1} yields the claim.
    %Thus $c$ is two-periodic.
\end{proof}

\noindent
If $a \neq b$, then for all $r \geq 1$ any $(r,b,a)$-covering in the king grid is two-periodic:

\begin{theorem} \label{kinggrid}
    Let $r \geq 1$ be arbitrary and let $D$ be the relative $r$-neighborhood of the king grid and assume that $a \neq b$.
    Then any $(D,b,a)$-covering is two-periodic.
    In other words, all $(r,b,a)$-coverings in the king grid are two-periodic whenever $a \neq b$.
\end{theorem}

\begin{proof}
    Let $c$ be an arbitrary $(D,b,a)$-covering.
    % By Theorem \ref{theorem on line polynomial factors} it is sufficient to show that $g = f_D -(b-a)$ has no line polynomial factors.
    The outer edges of $g = f_D -(b-a)$
    are in directions $(1,0),(-1,0),(0,1)$ and $(0,-1)$.
    Hence,
    by Lemma~\ref{lemma1} any line polynomial factor of $g$
    has direction
    $(1,0)$
    or
    $(0,1)$.
    Let $\vec{v} \in \{ (1,0),(0,1)\}$.
    We have
    $\phi_1 = 1+t+\ldots + t^{r-1} + (1-(b-a))t^r +t^{r+1} + \ldots + t^{2r} \in \mathcal{F}_{\vec{v}}(g)$
    and
    $\phi_2 = 1+t+ \ldots + t^{2r} \in \mathcal{F}_{\vec{v}}(g)$.
    See Figure \ref{Illustrations} for an illustration in the case $r=2$.
    Since $\phi_2-\phi_1=(b-a)t^r$ is a non-trivial monomial, $\phi_1$ and $\phi_2$ have no common factors.
    Thus, by Theorem \ref{theorem1}
    the polynomial
    $g$ has no line polynomial factors and the claim follows by Corollary \ref{corollary1}.
    %\qed
\end{proof}

\noindent
In the above proofs we used the fact that two Laurent polynomials in one variable have no common factors if and only if they generate the entire ideal $\C[t^{\pm1}]$, and they do this if and only if they generate a non-zero monomial.
This is known as the \emph{weak Nullstellensatz} \cite{cox}.

%\subsubsection*{General convex neighborhoods}

A shape $D \subseteq \Z^2$ is \emph{convex} if it is the intersection $D = \conv(D) \cap \Z^2$ where $\conv(D) \subseteq \R^2$ is the real convex hull of $D$.
Above all our shapes were convex.
Next we generalize the above theorems and give a sufficient condition for forced periodicity of $(D,b,a)$-coverings for convex $D$.

So, let $D \subseteq \Z^2$ be a finite convex shape.
Any $(D,b,a)$-covering has a periodizer $g = f_D - (b-a)$.
As earlier, we study whether $g$ has any line polynomial factors since if it does not, then Corollary \ref{corollary1} guarantees forced periodicity.
For any $\vec{v} \neq \vec{0}$ the set $\mathcal{F}_{\vec{v}}(f_D)$ contains only polynomials $\phi_n = 1 + \ldots + t^{n-1}$ for different $n \geq 1$ since $D$ is convex: if $D$ contains two points, then $D$ contains every point between them.
Thus, $\mathcal{F}_{\vec{v}}(g)$ contains only polynomials $\phi_n$ for different $n \geq 1$ and, if $b-a \neq 0$, it may also contain a polynomial $\phi_{n_0} - (b-a)t^{m_0}$ for some $n_0 \geq 1$ such that $\phi_{n_0} \in \mathcal{F}_{\vec{v}}(f_D)$ and for some $m_0 \geq 0$.
If $b-a=0$, then $g=f_D$ and thus $\mathcal{F}_{\vec{v}}(g) = \mathcal{F}_{\vec{v}}(f_D)$.

Two polynomials $\phi_m$ and $\phi_n$ have a common factor if and only if $\gcd(m,n) > 1$.
More generally, the polynomials $\phi_{n_1}, \ldots, \phi_{n_r}$ have a common factor if and only if $d = \gcd(n_1,\ldots,n_r) > 1$ and, in fact, their greatest common factor is the $d$th \emph{cyclotomic polynomial}
$$
\prod_{\substack{1 \leq k \leq d \\ \gcd(k,d) = 1}}(t - e^{i \cdot \frac{2 \pi k}{d}}).
$$

Let us introduce the following notation. For any polynomial $f$, we
denote by $\mathcal{F}'_{\vec{v}}(f)$ the set of normal forms of the
non-zero fibers $\sum_{k \in \Z} f_{\vec{u} + k \vec{v}} X^{\vec{u} + k \vec{v}}$ for all $\vec{u}\not\in\Z\vec{v}$.
In other words, we exclude the fiber through the origin. Let us also denote $\fib{\vec{v}}{f}$ for the
normal form of the fiber $\sum_{k \in \Z} f_{k \vec{v}} X^{k \vec{v}}$ through the origin. We have
$\mathcal{F}_{\vec{v}}(f)=\mathcal{F}'_{\vec{v}}(f)\cup\{\fib{\vec{v}}{f}\}$ if $\fib{\vec{v}}{f} \neq 0$ and $\mathcal{F}_{\vec{v}}(f)=\mathcal{F}'_{\vec{v}}(f)$ if $\fib{\vec{v}}{f} = 0$.

Applying Theorems \ref{theorem on line polynomial factors} and \ref{theorem1} we have the following theorem that gives sufficient conditions for every $(D,b,a)$-covering to be periodic for a finite and convex $D$.
This theorem generalizes the results proved above.
In fact, they are corollaries of the theorem.
The first part of the theorem was also mentioned in \cite{geravker-puzynina} in a slightly different context and in a more general form.

\begin{theorem} \label{convex perfect covering}
    Let $D$ be a finite convex shape, $g = f_D - (b-a)$
    and let $E$ be the set of the outer edge directions of $g$.
    %Assume that the following holds for all but one $\vec{v} \in E$.
    \begin{itemize}
    \item Assume that $b-a = 0$. For any $\vec{v} \in E$ denote $d_\vec{v}=\gcd(n_1,\ldots,n_r)$ where
    $\mathcal{F}_{\vec{v}}(g) = \{ \phi_{n_1},\ldots,\phi_{n_r}\}$. If $d_\vec{v} = 1$ holds for all
    $\vec{v} \in E$, then every $(D,b,a)$-covering is two-periodic. If $d_\vec{v} = 1$ holds for all but some
    parallel $\vec{v} \in E$, then every $(D,b,a)$-covering is periodic.
    \item Assume that $b-a \neq 0$. For any $\vec{v} \in E$ denote $d_\vec{v}=\gcd(n_1,\ldots,n_r)$ where
    $\mathcal{F}'_{\vec{v}}(g) = \{ \phi_{n_1},\ldots,\phi_{n_r}\}$.
    If the $d_\vec{v}$'th cyclotomic polynomial and $\fib{\vec{v}}{g}$ have no common factors for
    any $\vec{v} \in E$, then every $(D,b,a)$-covering is two-periodic. If the condition holds for all but some
    parallel $\vec{v} \in E$, then every $(D,b,a)$-covering is periodic. (Note that the condition is satisfied, in particular, if $d_\vec{v} = 1$.)
    \end{itemize}
\end{theorem}

\begin{proof}
Assume first that $b-a=0$.
If $d_{\vec{v}}=1$ for all $\vec{v} \in E$, then the $\vec{v}$-fibers of $g$ have no common factors and hence by Theorem \ref{theorem1} $g$ has no line polynomial factors.
If $d_\vec{v} = 1$ holds for all but some
parallel $\vec{v} \in E$, then all the line polynomial factors of $g$ are in parallel directions.
Thus, the claim follows by Theorem \ref{theorem on line polynomial factors}.

Assume then that $b-a \neq 0$.
If the $d_\vec{v}$'th cyclotomic polynomial and $\fib{\vec{v}}{g}$ have no common factors for all $\vec{v} \in E$, then by Theorem \ref{theorem1} $g$ has no line polynomial factors.
If the condition holds for all but some
parallel $\vec{v} \in E$, then all the line polynomial factors of $g$ are in parallel directions.
Thus, by Theorem \ref{theorem on line polynomial factors} the claim holds also in this case.
\end{proof}

%\section{Perfect colorings and configurations}
%\section{Bigger alphabets}

%\newpage

\section{Forced periodicity of perfect colorings over arbitrarily large alphabets}
\label{perfect colorings}

In this section we prove a theorem that gives a sufficient condition for forced periodicity of two-dimensional perfect colorings over an arbitrarily large alphabet.
As corollaries of the theorem and theorems from the previous section we obtain conditions for forced periodicity of perfect colorings in two-dimensional infinite grid graphs.

We start by proving some lemmas that work in any dimension.
We consider the vector presentations of perfect colorings because this way we get a non-trivial annihilator for any such vector presentation:

\begin{lemma} \label{apulemma1}
    Let $c$ be the vector presentation of a $D$-perfect coloring over an alphabet of size $n$ with matrix $\mathbf{B} = (b_{ij})_{n \times n}$.
    Then $c$ is annihilated by
    the polynomial
    $$
    f(X) = \sum_{\vec{u} \in D} \mathbf{I} X^{-\vec{u}} - \mathbf{B}.
    $$
\end{lemma}

\noindent
\emph{Remark.}
Note the similarity of the above annihilator to the periodizer $\sum_{\vec{u}\in D} X^{-\vec{u}} - (b-a)$ of a $(D,b,a)$-covering.

\begin{proof}
    Let $\vec{v} \in \Z^d$ be arbitrary and assume that $c_{\vec{v}} = \vec{e}_j$.
    Then $(\mathbf{B} c)_{\vec{v}} = \mathbf{B} \vec{e}_j$ is the $j$th column of $\mathbf{B}$.
    On the other hand, from the definition of $\mathbf{B}$ we have $((\sum_{\vec{u} \in D} \mathbf{I} X^{-\vec{u}})  c)_{\vec{v}} = \sum_{\vec{u} \in D} c_{\vec{v} + \vec{u}} = \sum _{i=1}^n b_{ij} \vec{e}_i$ which is also the $j$th column of $\mathbf{B}$.
    Thus, $(fc)_{\vec{v}} = 0$ and hence $fc = 0$ since $\vec{v}$ was arbitrary.
\end{proof}

\noindent
The following lemma shows that as in the case of integral configurations with non-trivial annihilators, also the vector presentation of a perfect coloring has a special annihilator which is a product of difference polynomials.
By congruence of two polynomials with integer matrices as coefficients (mod $p$) we mean that their corresponding coefficients are congruent (mod $p$) and by congruence of two integer matrices (mod $p$) we mean that their corresponding components are congruent (mod $p$).

\begin{lemma} \label{apulemma2}
    Let $c$ be the vector presentation of a $D$-perfect coloring over an alphabet of size $n$  with matrix $\mathbf{B} = (b_{ij})_{n \times n}$.
    Then $c$ is annihilated by
    the polynomial
    $$
    g(X) = (\mathbf{I} X^{\vec{v}_1} - \mathbf{I}) \cdots (\mathbf{I} X^{\vec{v}_m} - \mathbf{I})
    $$
    for some vectors $\vec{v}_1,\ldots,\vec{v}_m$.
\end{lemma}

\begin{proof}
    By Lemma \ref{apulemma1} the power series $c$ is annihilated by $f(X) = \sum_{\vec{u} \in D} \mathbf{I} X^{-\vec{u}} - \mathbf{B}$.
    Let $p$ be a prime larger than $n c_{\text{max}}$ where $c_{\text{max}}$ is the maximum absolute value of the components of the coefficients of $c$.
    Since the coefficients of $f$ commute with each other, we have for any positive integer $k$ using the binomial theorem that
    \begin{equation*} \label{equation2}
    f^{p^k} = f^{p^k}(X) \equiv \sum_{\vec{u} \in D} \mathbf{I} X^{-p^k \vec{u}} - \mathbf{B}^{p^k} \ \ (\text{mod } p).
    \end{equation*}
    % We denote
    % $
    % f_{p^k} = \sum_{\vec{u} \in D} \mathbf{I} X^{-p^k \vec{u}} - \mathbf{B}^{p^k}.
    % $
    We have $f^{p^k}(X) c(X) \equiv 0 \ \ (\text{mod } p)$.
    There are only finitely many distinct matrices $\mathbf{B}^{p^k} \ \ (\text{mod } p)$.
    So, let $k$ and $k'$ be distinct and such that $\mathbf{B}^{p^{k}} \equiv \mathbf{B}^{p^{k'}} \ \ (\text{mod } p)$.
    Then the coefficients of
    $f' = f^{p^{k}} - f^{p^{k'}} \ \ (\text{mod } p)$
    are among $\mathbf{I}$ and $- \mathbf{I}$.
    Since $f^{p^{k}} c \equiv 0 \ \ (\text{mod } p)$ and $f^{p^{k'}} c \equiv 0 \ \ (\text{mod } p)$, also
    $$f' c \equiv 0 \ \ (\text{mod } p).$$
    The components of the configuration $f' c$ are bounded in absolute value by $n c_{\text{max}}$.
    Since we chose $p$ larger than $n c_{\text{max}}$,
    this implies that
    $$
    f'c = 0.
    $$

    Because $f' = \sum_{\vec{u} \in P_1} \mathbf{I} X^{\vec{u}} - \sum_{\vec{u} \in P_2} \mathbf{I} X^{\vec{u}}$ for some finite subsets $P_1$ and $P_2$ of $\Z^d$, the annihilation of $c$ by $f'$ is equivalent to the annihilation of every layer of $c$ by $f'' = \sum_{\vec{u} \in P_1} X^{\vec{u}} - \sum_{\vec{u} \in P_2} X^{\vec{u}}$.
    Thus, every layer of $c$ has a non-trivial annihilator and hence
    by Theorem \ref{special annihilator} every layer of $c$ has a special annihilator which is a product of difference polynomials.
    Let
    $$
    g' = (X^{\vec{v}_1}-1) \cdots (X^{\vec{v}_m}-1)
    $$
    be the product of all these special annihilators.
    Since $g'$ annihilates every layer of $c$,
    the polynomial
    $$
    g = (\mathbf{I} X^{\vec{v}_1} - \mathbf{I}) \cdots (\mathbf{I} X^{\vec{v}_m} - \mathbf{I})
    $$
    annihilates $c$.
\end{proof}

%\noindent
%For the proof of Theorem \ref{main theorem} we need one more lemma concerning additive cellular automata.

\begin{lemma} \label{surjecticity lemma}
    %Let $R = \Z_p^{n \times n}$ and let $M = \Z_p^n$.
    Let $p$ be a prime and let $H$ be an additive CA over $\Z_p^n$ determined by a polynomial $h = \sum_{i=0}^k \mathbf{A}_i X^{\vec{u}_i} \in \Z_p^{n \times n}[X^{\pm1}]$ whose coefficients $\mathbf{A}_i$ commute with each other.
    Assume that there exist $M \in \Z_p \setminus \{0\}$ and matrices $\mathbf{C}_0, \ldots , \mathbf{C}_k$ that commute with each other and with every $\mathbf{A}_i$ such that
    $$
    \mathbf{C}_0 \mathbf{A}_0 + \ldots + \mathbf{C}_k \mathbf{A}_k = M \cdot \mathbf{I}
    $$
    holds in $\Z_p^{k \times k}$.
    Then $H$ is surjective.
\end{lemma}

\begin{proof}
    Assume the contrary that $H$ is not surjective.
    By the Garden-of-Eden theorem $H$ is not pre-injective
    and hence there exist two distinct asymptotic configurations $c_1$ and $c_2$ such that $H(c_1) = H(c_2)$, that is, $h(X) c_1(X) = h(X) c_2(X)$.
    Thus, $h$ is an annihilator of $e=c_1-c_2$. %and hence $H(e)=0$.
    Without loss of generality we may assume that $c_1(\vec{0}) \neq c_2(\vec{0})$, {\it i.e.}, that $e(\vec{0}) = \vec{v} \neq \vec{0}$.
    Let $l$ be such that the support $\supp(e) = \{ \vec{u} \in \Z^d \mid e(\vec{u}) \neq \vec{0} \}$ of $e$ is contained in a $d$-dimensional $p^l \times \ldots \times p^l$ hypercube.
    Note that in $\Z_p^{k \times k}$ we have
    $$
    f^{p^l} = \sum_{i=0}^k \mathbf{A}_i^{p^l} X^{p^l \vec{u}_i}
    $$
    which is also an annihilator of $e$.
    Hence, by the choice of $l$ we have $\mathbf{A}_i^{p^l} \vec{v} = \vec{0}$ for all $i \in \{ 1 , \ldots , k \}$.
    By raising the identity
    $$
    \mathbf{C}_0 \mathbf{A}_0 + \ldots + \mathbf{C}_k \mathbf{A}_k = M \cdot \mathbf{I}
    $$
    to power $p^l$ and multiplying the result by the vector $\vec{v}$ from the right we get
    $$
    M^{p^l} \cdot \vec{v} = \mathbf{C}_0^{p^l} \mathbf{A}_0^{p^l} \vec{v} + \ldots + \mathbf{C}_k^{p^l} \mathbf{A}_k^{p^l} \vec{v}
    = \vec{0} + \ldots + \vec{0} = \vec{0}.
    $$
    However, this is a contradiction because $M^{p^l} \vec{v} \neq \vec{0}$.
    Thus, $H$ must be surjective as claimed.
\end{proof}

\begin{theorem} \label{main theorem}
    Let $D \subseteq \Z^2$ be a finite shape
    and assume that there exists an integer $t_0$ such that the polynomial
    $
    f_D - t = \sum_{\vec{u} \in D} X^{-\vec{u}} - t
    $
    has no line polynomial factors whenever $t \neq t_0$.
    Then any $D$-perfect coloring with matrix $\mathbf{B}$ is two-periodic whenever $\det(\mathbf{B} - t_0 \mathbf{I}) \neq 0$.
    If $f_D - t$ has no line polynomial factors for any $t$, then every $D$-perfect coloring is two-periodic.
\end{theorem}

\begin{proof}
Let $c$ be the vector presentation of a $D$-perfect coloring with matrix $\mathbf{B}$.
By Lemmas \ref{apulemma1} and \ref{apulemma2} it has two distinct annihilators:
$f= \sum_{\vec{u} \in D} \mathbf{I} X^{-\vec{u}} - \mathbf{B}$
and
$g= (\mathbf{I} X^{\vec{v}_1} - \mathbf{I}) \cdots (\mathbf{I} X^{\vec{v}_m} - \mathbf{I})$.
Let us replace $\mathbf{I}$ by 1 and $\mathbf{B}$ by a variable $t$ and consider the corresponding integral polynomials $f' = \sum_{\vec{u} \in D} X^{-\vec{u}} - t = f_D - t$ and $g' = (X^{\vec{v}_1} - 1) \cdots (X^{\vec{v}_m} - 1)$ in $\C[x,y,t]$.
Here $X=(x,y)$.

Without loss of generality we may assume that $f'$ and $g'$ are proper polynomials.
Indeed, we can multiply $f'$ and $g'$ by monomials such that the obtained polynomials $f''$ and $g''$ are proper polynomials and that they have a common factor if and only if $f'$ and $g'$ have a common factor.
So, we may consider $f''$ and $g''$ instead of $f'$ and $g'$ if they are not proper polynomials.

% that have no common factors in $\C[X]$ if and only if the original Laurent polynomials have no common factors in $\C[X^{\pm1}]$ since we can multiply them by suitable monomials

We consider the $y$-resultant $\Res_y(f',g')$ of $f'$ and $g'$, and write
$$
\Res_y(f',g') = f_0(t) + f_1(t)x + \ldots + f_k(t) x^k.
$$
By the properties of resultants $\Res_y(f',g')$ is in the ideal generated by $f'$ and $g'$, and it can be the zero polynomial only if $f'$ and $g'$ have a common factor.
Since $g'$ is a product of line polynomials, any common factor of $f'$ and $g'$ is also a product of line polynomials.
In particular, if $f'$ and $g'$ have a common factor, then they have a common line polynomial factor.
However, by the assumption $f'$ has no line polynomial factors if $t \neq t_0$.
Thus, $f'$ and $g'$ may have a common factor only if $t=t_0$ and hence $\Res_y(f',g')$ can be zero only if $t = t_0$.
On the other hand, $\Res_y(f',g') = 0$ if and only if $f_0(t) = \ldots = f_k(t) = 0$.
We conclude that $\gcd(f_0(t), \ldots , f_k(t)) = (t-t_0)^m$ for some $m \geq 0$.
Thus,
$$
\Res_y(f',g') = (t-t_0)^m (f'_0(t) + f'_1(t)x + \ldots + f'_k(t) x^k)
$$
where the polynomials $f'_0(t), \ldots , f'_k(t)$ have no common factors.

By the Euclidean algorithm there are polynomials $a_0(t), \ldots , a_k(t)$
such that
\begin{equation} \label{polynomial identity}
    a_0(t) f_0'(t) + \ldots + a_k(t) f_k'(t) = 1.
\end{equation}
Moreover, the coefficients of the polynomials $a_0(t), \ldots , a_k(t)$ are rational numbers because the polynomials $f'_0(t), \ldots , f'_k(t)$ are integral.
Note that if $f'$ has no line polynomial factors for any $t$, then $m=0$ and hence $f_i'(t) = f_i(t)$ for every $i \in \{1,\ldots,k\}$.

Let us now consider the polynomial
$$
(\mathbf{B}-t_0 \mathbf{I})^m(f_0'(\mathbf{B}) + f_1'(\mathbf{B}) x + \ldots + f'_k(\mathbf{B}) x^k)
$$
which is obtained from $\Res_y(f',g')$ by plugging back $\vec{I}$ and $\vec{B}$ in the place of $1$ and $t$, respectively.
Since $\Res_y(f',g')$ is in the ideal generated by $f'$ and $g'$, the above polynomial is in the ideal generated
by $f$ and $g$.
Thus, it is an annihilator of $c$ because both $f$ and $g$ are annihilators of $c$.

Assume that $\det(\mathbf{B} - t_0 \mathbf{I}) \neq 0$ or that $m=0$.
%or that $f'$ has no line polynomial factors for any $t$.
Now also
$$
h = f_0'(\mathbf{B}) + f_1'(\mathbf{B}) x + \ldots + f'_k(\mathbf{B}) x^k
$$
is an annihilator of $c$.
Since
$f'_0(t), \ldots , f'_k(t)$ have no common factors, $h$ is non-zero, because otherwise it would be $f_0'(\mathbf{B}) = \ldots = f_k'(\mathbf{B}) = 0$ and the minimal polynomial of $\mathbf{B}$ would be a common factor of $f'_0(t), \ldots , f'_k(t)$, a contradiction.

Plugging $t = \mathbf{B}$ to Equation \ref{polynomial identity}
we get
\begin{equation*}
    a_0(\mathbf{B}) f_0'(\mathbf{B}) + \ldots + a_k(\mathbf{B}) f_k'(\mathbf{B}) = \mathbf{I}.
\end{equation*}
Let us multiply the above equation by a common multiple $M$ of all the denominators of the rational numbers appearing in the equation and let us consider it (mod $p$) where $p$ is a prime that does not divide $M$.
We obtain the following identity
\begin{equation*} \label{matrix identity}
    a_0'(\mathbf{B}) f_0'(\mathbf{B}) + \ldots + a_k'(\mathbf{B}) f_k'(\mathbf{B}) = M \cdot \mathbf{I} \not \equiv 0 \ \ (\text{mod } p)
\end{equation*}
where all the coefficients in the equation are integer matrices.

By Lemma \ref{surjecticity lemma} the additive CA determined by $h = \sum_{i=0}^k f_i'(\mathbf{B})x^i$ is surjective.
Since $h$ is a polynomial in variable $x$ only, it defines a 1-dimensional CA $H$ which is surjective and which maps every horizontal fiber of $c$ to 0.
Hence, every horizontal fiber of $c$ is a pre-image of 0.
Let $c'$ be a horizontal fiber of $c$.
The Garden-of-Eden theorem implies that $0$ has finitely many, say $N$, pre-images under $H$.
Since also every translation of $c'$ is a pre-image of $0$, we conclude that $c' = \tau^i(c')$ for some $i \in \{0, \ldots , N-1\}$.
Thus, $(N-1)!$ is a common period of all the horizontal fibers of $c$ and hence $c$ is horizontally periodic.

Repeating the same argumentation for the $x$-resultant of $f'$ and $g'$ we can show that $c$ is also vertically periodic.
Thus, $c$ is two-periodic.
\end{proof}

\noindent
As corollaries of the above theorem and theorems from the previous section, we obtain new proofs for forced periodicity of perfect colorings in the square and the triangular grids, and a new result for forced periodicity of perfect colorings in the king grid:

\begin{corollary}[\cite{puzynina2}] \label{general square grid1}
    Let $D$ be the relative 1-neighborhood of the square grid. Then any $D$-perfect coloring with matrix $\mathbf{B}$ is two-periodic whenever $\det(\mathbf{B} - \mathbf{I}) \neq 0$.
    In other words, any $1$-perfect coloring with matrix $\mathbf{B}$ in the square grid is two-periodic whenever $\det(\mathbf{B} - \mathbf{I}) \neq 0$.
\end{corollary}

\begin{proof}
    In our proof of Theorem \ref{squaregrid1} it was shown that the polynomial $f_{D} - t$ has no line polynomial factors if $t \neq 1$.
    Thus, by Theorem \ref{main theorem} any $(D, \mathbf{B})$-coloring is two-periodic whenever
    $\det(\mathbf{B} - \mathbf{I}) \neq 0$.
\end{proof}

\begin{corollary}[\cite{puzynina2}] \label{general square grid2}
    Let $D$ be the relative 1-neighborhood of the triangular grid. Then any $D$-perfect coloring with matrix $\mathbf{B}$ is two-periodic whenever $\det(\mathbf{B} + \mathbf{I}) \neq 0$.
    In other words, any $1$-perfect coloring with matrix $\mathbf{B}$ in the triangular grid is two-periodic whenever $\det(\mathbf{B} + \mathbf{I}) \neq 0$.
\end{corollary}

\begin{proof}
    In the proof of Theorem \ref{triangulargrid1} it was shown that the polynomial $f_{D} - t$ has no line polynomial factors if $t \neq -1$.
    Thus, by Theorem \ref{main theorem} any $(D, \mathbf{B})$-coloring is two-periodic whenever
    $\det(\mathbf{B} + \mathbf{I}) \neq 0$.
\end{proof}

\begin{corollary}[\cite{puzynina2}] \label{general triangular grid1}
    Let $r \geq 2$ and let $D$ be the relative $r$-neighborhood of the square grid. Then every $D$-perfect coloring is two-periodic.
    In other words, any $r$-perfect coloring in the square grid is two-periodic for all $r \geq 2$.
\end{corollary}

\begin{proof}
    In the proof of Theorem \ref{squaregrid2} it was shown that the polynomial $f_D - t$ has no line polynomial factors for any $t$.
    Thus, by Theorem \ref{main theorem} every $D$-perfect coloring is two-periodic.
\end{proof}

\begin{corollary}[\cite{puzynina2}] \label{general triangular grid2}
    Let $r \geq 2$ and let $D$ be the relative $r$-neighborhood of the triangular grid. Then every $D$-perfect coloring is two-periodic.
    In other words, any $r$-perfect coloring in the triangular grid is two-periodic for all $r \geq 2$.
\end{corollary}

\begin{proof}
    In the proof of Theorem \ref{triangulargrid2} it was shown that the polynomial $f_D - t$ has no line polynomial factors for any $t$.
    Thus, by Theorem \ref{main theorem} every $D$-perfect coloring is two-periodic.
\end{proof}

\begin{corollary} \label{general king grid}
    Let $r \geq 1$ and let $D$ be the relative $r$-neighborhood of the king grid. Then every $D$-perfect coloring with matrix $\mathbf{B}$ is two-periodic whenever $\det(\mathbf{B}) \neq 0$.
    In other words, every $r$-perfect coloring with matrix $\mathbf{B}$ in the king grid is two-periodic whenever $\det(\mathbf{B}) \neq 0$.
\end{corollary}

\begin{proof}
    In the proof of Theorem \ref{kinggrid} we showed that the polynomial $f_D - t$ has no line polynomial factors if $t \neq 0$.
    Thus, by Theorem \ref{main theorem} any $(D,\mathbf{B})$-coloring is two-periodic whenever $\det(\mathbf{B}) \neq 0$.
\end{proof}

\noindent
\emph{Remark.}
Note that the results in Corollaries \ref{general square grid1}, \ref{general square grid2}, \ref{general triangular grid1} and \ref{general triangular grid2} were stated and proved in \cite{puzynina2} in a slightly more general form.
Indeed,
in \cite{puzynina2} it was proved that if a configuration $c \in \A^{\Z^2}$ is annihilated by
$$
\sum_{\vec{u} \in D} \mathbf{I} X^{-\vec{u}} - \mathbf{B}
$$
where $\mathbf{B} \in \Z^{n \times n}$ is an arbitrary integer matrix whose determinant satisfies the conditions in the four corollaries and $D$ is as in the corollaries, then $c$ is necessarily periodic.
This kind of configuration was called a \emph{generalized centered function}.
However, in Lemma \ref{apulemma1} we proved that the vector presentation of any $D$-perfect coloring with matrix $\mathbf{B}$
is annihilated by this polynomial, that is, we proved that the vector presentation of a perfect coloring is a generalized centered function.
By analyzing the proof of Theorem \ref{main theorem} we see that the theorem holds also for generalized centered functions and hence the corollaries following it hold also for generalized centered functions, and thus we have the same results as in \cite{puzynina2}.

%\newpage

\section{Forced periodicity of configurations of low abelian complexity} \label{abelian}

In this section we prove a statement concerning forced periodicity of two-dimensional configurations of low abelian complexity which generalizes a result in \cite{geravker-puzynina}.
In fact, as in \cite{geravker-puzynina} we generalize the definition of abelian complexity from finite patterns to polynomials and prove a statement of forced periodicity under this more general definition of abelian complexity.

%Clearly a configuration with low abelian complexity with respect to a finite shape $D$ is a perfect coloring

Let $c \in \{\vec{e}_1, \ldots , \vec{e}_n \}^{\Z^d}$ and let $D \subseteq \Z^d$ be a finite shape.
Consider the polynomial
$f=\mathbf{I} \cdot f_D(X) = \sum_{\vec{u} \in D} \mathbf{I} X^{-\vec{u}} \in \Z^{n \times n}[X^{\pm1}]$.
The $i$th coefficient of
$(fc)_{\vec{v}} = \sum_{\vec{u} \in D} \mathbf{I} \cdot \vec{c_{\vec{v}+\vec{u}}}$
tells the number of cells of color $\vec{e}_i$ in the $D$-neighborhood of $\vec{v}$ in $c$ and hence
the abelian complexity of $c$ with respect to $D$ is exactly the number of distinct coefficients of $fc$.
\iffalse
Thus
$$
A(c,D) = | \{ (fc)_{\vec{v}} \mid \vec{v} \in \Z^2 \}|.
$$
\fi

More generally, we define the abelian complexity $A(c,f)$ of an integral vector configuration $c \in \A ^{\Z^d}$ where $\A$ is finite set of integer vectors \emph{with respect to a polynomial $f \in \Z^{n \times n}[X^{\pm1}]$} as
$$
A(c,f) = | \{ (fc)_{\vec{v}} \mid \vec{v} \in \Z^d \}|.
$$
This definition can be extended to integral configurations and polynomials.
Indeed,
we define the abelian complexity $A(c,f)$ of a configuration $c \in \A^{\Z^d}$ where $\A \subseteq \Z$
with respect to a polynomial $f = \sum f_i X^{\vec{u}_i} \in \Z[X^{\pm1}]$
to be the abelian complexity $A(c',f')$ of the vector presentation $c'$ of $c$ with respect to the polynomial $f' = \mathbf{I} \cdot f = \sum f_i \cdot \mathbf{I} \cdot X^{\vec{u}_i}$.
Consequently,
we say that $c$ has low abelian complexity with respect to a polynomial $f$ if $A(c,f) = 1$.
Clearly this definition is consistent with the definition of low abelian complexity of a configuration with respect to a finite shape since
if $c$ is an integral configuration, then
$A(c,D)=1$ if and only if $A(c,f_D)=1$, and if $c$ is an integral vector configuration, then $A(c,D)=1$ if and only if $A(c,\mathbf{I} \cdot f_D)=1$.

We study forced periodicity of two-dimensional configurations of low abelian complexity.
Note that a configuration of low abelian complexity is not necessarily periodic.
Indeed, in \cite{puzynina3} it was shown that there exist non-periodic two-dimensional configurations that have abelian complexity $A(c,D) = 1$ for some finite shape $D$.
However, in \cite{geravker-puzynina}
it was shown that if $A(c,f) = 1$ and if the polynomial $f$ has no line polynomial factors, then $c$ is two-periodic assuming that the support of $f$ is convex.
The following theorem strengthens this result and shows that the convexity assumption of the support of the polynomial is not needed.
We obtain this result as a corollary of Theorem \ref{theorem on line polynomial factors}.

\begin{theorem} \label{abelian complexity result}
    \label{theorem abelian}
	Let $c$ be a two-dimensional integral configuration over an alphabet of size $n$ and assume that it has low abelian complexity with respect to a polynomial $f \in \Z[x^{\pm1}, y^{\pm1}]$.
	If $f$ has no line polynomial factors, then $c$ is two-periodic.
	If $f$ has line polynomial factors in a unique primitive direction $\vec{v}$, then $c$ is $\vec{v}$-periodic.
    Thus, if $f_D$ has no line polynomial factors or its line polynomial factors are in a unique primitive direction, then any configuration that has low abelian complexity with respect to $D$ is two-periodic or periodic, respectively.
\end{theorem}

\begin{proof}
    By the assumption that $A(c,f)=1$ we have $f'c' = \vec{c}_0 \mathbbm{1}$ for some $\vec{c}_0 \in \Z^n$
    where $c'$ is the vector presentation of $c$ and $f' = \mathbf{I} \cdot f$.
    Thus,
    $f$ periodizes every layer of $c'$.
    If $f$ has no line polynomial factors,
    then by Theorem \ref{theorem on line polynomial factors} every layer of $c'$ is two-periodic and hence $c'$ is two-periodic.
    If $f$ has line polynomial factors in a unique primitive direction $\vec{v}$, then by Theorem \ref{theorem on line polynomial factors} every layer of $c'$ is $\vec{v}$-periodic and hence also $c'$ is $\vec{v}$-periodic.
    Since $c$ is periodic if and only if its vector presentation $c'$ is periodic, the claim follows.
\end{proof}

\noindent
\emph{Remark.}
In \cite{geravker-puzynina} a polynomial $f \in \Z[X^{\pm1}]$ is called abelian rigid if an integral configuration $c$ having low abelian complexity with respect to $f$ implies that $c$ is strongly periodic.
In the above theorem we proved that if a polynomial $f \in \Z[x^{\pm1},y^{\pm1}]$ has no line polynomial factors then it is abelian rigid.
Also, the converse holds as proved in \cite{geravker-puzynina}, that is, if a polynomial $f \in \Z[x^{\pm1},y^{\pm1}]$ has a line polynomial factor then it is not abelian rigid.
This means that if $f$ has a line polynomial factor then there exists a configuration which is not two-periodic but has low abelian complexity with respect to $f$.
In fact this direction holds for all $d$, not just for $d=2$ as reported in \cite{geravker-puzynina}.

In the following example we introduce an open problem related to configurations of low abelian complexity.

\begin{example}[Periodic tiling problem]
    This example concerns \emph{translational tilings} by a single tile.
    In this context by a tile we mean any finite subset $F \subseteq \Z^d$ and by a tiling by the tile $F$ we mean such subset $C \subseteq \Z^d$ that every point of the grid $\Z^d$ has a unique presentation as a sum of an element of $F$ and an element of $C$.
    Presenting the tiling $C$ as its indicator function we obtain a $d$-dimensional binary configuration $c \in \{0,1\}^{\Z^d}$ defined by
    $$
    c_{\vec{u}} =
    \begin{cases}
        1, \text{ if } \vec{u} \in C \\
        0, \text{ if } \vec{u} \not \in C
    \end{cases}.
    $$
    The configuration $c$ has exactly $|F|$ different patterns of shape $-F$, namely the patterns with exactly one symbol 1.
    In other words, it has low complexity with respect to $-F$.
    Let $f = f_F = \sum_{\vec{u} \in F} X^{-\vec{u}}$ be the characteristic polynomial of $F$.
    Since $C$ is a tiling by $F$, we have $fc = \mathbbm{1}$.
    In fact, $c$ has low abelian complexity with respect to $f$ and $-F$.
    Thus, by Theorem \ref{abelian complexity result} any tiling by $F \subset \Z^2$ is two-periodic if $f_F$ has no line polynomial factors.

    The periodic tiling problem claims that if there exists a tiling by a tile $F \subseteq \Z^d$, then there exists also a periodic tiling by $F$ \cite{lagarias-wang, Szegedy1998}.
    By a simple pigeonholing argument it can be seen that in dimension $d=1$ all translational tilings by a single tile are periodic and hence the periodic tiling problem holds in dimension 1 \cite{newman}.
    For $d \geq 2$ the conjecture is much trickier and only recently it was proved by Bhattacharya that it holds for $d=2$ \cite{bhattacharya}.
    In \cite{greenfeld-tao} it was presented a slightly different proof in the case $d=2$ with some generalizations.
    For $d \geq 3$ the conjecture is still partly open.
    However,
    very recently it has been proved that for some sufficiently large $d$ the periodic tiling conjecture is false \cite{greenfeld-tao2022}.

\end{example}

\section{Algorithmic aspects} \label{algortihmic aspects}

%Let $D \subseteq \Z^2$ be a finite shape and let $a$ and $b$ be non-negative integers.

All configurations in a subshift are periodic, in particular, if there are no configurations in the subshift at all!
It is useful to be able to detect such trivial cases.
%In this section

The set
$$
\mathcal{S}(D,b,a) =
\{ c \in \{0,1\}^{\Z^2} \mid (f_D - (b-a)) c = a \mathbbm{1}(X) \}
$$
of all $(D,b,a)$-coverings is an SFT for any given finite shape $D$ and non-negative integers $b$ and $a$.
Hence, the question whether there exist any $(D,b,a)$-coverings for a given neighborhood
$D$ and covering constants $b$ and $a$ is equivalent to the question whether
the SFT $\mathcal{S}(D,b,a)$ is non-empty.
The question of emptiness
of a given SFT is undecidable in general, but if the SFT is known to be not aperiodic, then
the problem becomes decidable as a classic argumentation by Hao Wang shows:

\begin{lemma}[\cite{wang}] \label{lemma wang}
    If an SFT is either the empty set or it contains a strongly periodic configuration, then its emptiness problem is decidable, that is, there is an algorithm to determine whether there exist any configurations in the SFT.
\end{lemma}

\noindent
In particular, if $g = f_D - (b-a)$ has line polynomial factors in at most one direction, then
the question whether there exist any $(D,b,a)$-coverings is decidable:

\iffalse
\begin{theorem}[\cite{wang}]
    If an SFT is either empty or it contains a periodic configuration then its emptiness problem is decidable, that is, there is an algorithm to determine whether there exist any configurations in the SFT.
\end{theorem}
\fi

%We have seen that if $g = f_D - (b-a)$ has line polynomial factors in at most one direction then every configuration in $\mathcal{S}$ is periodic (Theorem \ref{theorem on line polynomial factors}).
%This gives us the following decidability result.

%We will show that if $g = f_D - (b-a)$ has line polynomial factors in at most one direction then the question is decidable.
%Moreover we will show that if $g$ has no line polynomial factors then $\mathcal{S}$
%is a finite set and we will describe how to effectively construct all the elements of $\mathcal{S}$.

\begin{theorem}
\label{thm:effective}
Let a finite $D \subseteq \Z^2$ and non-negative integers $b$ and $a$ be given such that
the polynomial $g=f_D - (b-a) \in \Z[x^{\pm1},y^{\pm1}]$ has line polynomial factors in at most one primitive direction.
Then there exists an algorithm to determine whether there exist any $(D,b,a)$-coverings.
\end{theorem}

\begin{proof}	
Let $\mathcal{S}=\mathcal{S}(D,b,a)$ be the SFT of all $(D,b,a)$-coverings.
Since $g$ has line polynomial factors in at most one primitive direction, by Theorem \ref{theorem on line polynomial factors}
every element of $\mathcal{S}$ is periodic.
Any two-dimensional SFT that contains periodic configurations contains also two-periodic configurations. Thus,
$\mathcal{S}$ is either empty or contains a two-periodic configuration and hence by Lemma \ref{lemma wang} there is an algorithm to determine whether $\mathcal{S}$ is non-empty.
% By a standard argumentation by H. Wang \cite{wang}
% there exist semi-algorithms to determine whether a given SFT is empty and whether a given SFT contains a two-periodic configuration.	Running these two semi-algorithms in parallel gives us an algorithm to test whether $\mathcal{S} \neq \emptyset$.
%\qed
\end{proof}

\noindent
One may also want to design a perfect $(D,b,a)$-covering for given $D$, $b$ and $a$.
This can be effectively done under the assumptions of Theorem~\ref{thm:effective}:
As we have seen, if
$\mathcal{S}=\mathcal{S}(D,b,a)$ is non-empty, it contains a two-periodic configuration. For any two-periodic
configuration $c$ it is easy to check if $c$ contains a forbidden pattern.
%, so that there is an algorithm to
%test whether $c$ is in $\mathcal{S}$.
By enumerating two-periodic configurations one-by-one one is guaranteed to find eventually one that is in $\mathcal{S}$.

If the polynomial $g$ has no line polynomial factors, then the following stronger result holds:

\begin{theorem} \label{effective}
	If the polynomial $g=f_D - (b-a)$ has no line polynomial factors
	for given finite shape $D \subseteq \Z^2$ and non-negative integers $b$ and $a$, then the SFT $\mathcal{S} = \mathcal{S}(D,b,a)$ is finite. One can then effectively construct
    all the finitely many elements of $\mathcal{S}$.
\end{theorem}

\noindent
The proof of the first part of above theorem relies on the fact that
a two-dimensional subshift is finite if and only if it contains only two-periodic configurations ~\cite{ballier}.
If $g$ has no line polynomial factors, then every configuration it periodizes (including every configuration in $\mathcal{S}$) is two-periodic by Theorem \ref{theorem on line polynomial factors}, and hence $\mathcal{S}$ is finite.
The second part of the theorem, {\it i.e.},  the fact that one can
effectively produce all the finitely many elements of $\mathcal{S}$
holds generally for finite  SFTs in any dimension:
%(The proof is provided in the Appendix for the sake of completeness.)

\begin{lemma}
    \label{lemma effectively}
Given a finite $F \subseteq \A^*$ such that $X_F$ is finite, one can effectively construct the elements of $X_F$.
\end{lemma}

\begin{proof}	
Given a finite $F \subseteq \A^*$ and a pattern $p\in \A^D$, assuming that strongly periodic configurations are dense in $X_F$,
one can effectively check whether $p\in \Lang{X_F}$. Indeed, we have a semi-algorithm for the positive instances that
guesses a strongly periodic configuration $c$ and verifies that $c\in X_F$ and $p\in\Lang{c}$. A semi-algorithm for the
negative instances exists for any SFT $X_F$ and is a standard compactness argument: guess a finite $E\subseteq \Z^d$ such that
$D\subseteq E$ and verify that every $q\in \A^E$ such that $q|_D=p$ contains a forbidden subpattern.

Consequently, given finite $F,G \subseteq \A^*$,  assuming that strongly periodic configurations are dense in $X_F$ and $X_G$,
one can effectively determine whether $X_F=X_G$. Indeed, $X_F\subseteq X_G$ if and only if no $p\in G$ is in $\Lang{X_F}$, a
condition that we have shown above to be decidable. Analogously we can test $X_G\subseteq X_F$.

Finally, let a finite $F \subseteq \A^*$ be given such that $X_F$ is known to be finite. All elements of $X_F$ are strongly periodic so that strongly periodic configurations are certainly dense in $X_F$. One can effectively enumerate
all finite sets $P$ of strongly periodic configurations. For each $P$ that is translation invariant
(and hence a finite SFT) one can construct a finite set $G\subseteq \A^*$ of forbidden patterns such that $X_G=P$.
As shown above, there is an algorithm to test
whether $X_F=X_G=P$. Since $X_F$ is finite, a set $P$ is eventually found such that $X_F=P$.
%\qed
\end{proof}

Let us now turn to the more general question of existence of perfect colorings over alphabets of arbitrary size.
Let $D \subseteq \Z^2$ be a finite shape and let $\mathbf{B}$ be an $n \times n$ integer matrix.
To determine whether there exist any $(D,\mathbf{B})$-colorings is equivalent to asking whether the SFT
$$
\mathcal{S}(D,\mathbf{B}) = \{ c \in \{ \vec{e}_1,\ldots,\vec{e}_n \}^{\Z^2} \mid g c = 0 \}
$$
is non-empty where $g = \sum_{\vec{u} \in D} \mathbf{I} X^{-\vec{u}} - \mathbf{B}$ since it is exactly the set of the vector presentations of all $(D,\mathbf{B})$-colorings.

%\subsection{Abelian...}

\begin{theorem}
    Let a finite shape $D \subseteq \Z^2$, a non-negative integer matrix $\mathbf{B}$ and an integer $t_0$ be given such that the polynomial
    $f_D(x,y) - t \in \Z[x^{\pm1},y^{\pm1}]$
    has no line polynomial factors whenever $t \neq t_0$ and $\det(\mathbf{B} - t_0 \mathbf{I}) \neq 0$.
    Then there are only finitely many $(D,\mathbf{B})$-colorings and one can effectively construct them.
    In particular, there is an algorithm to determine whether there exist any $(D,\mathbf{B})$-colorings.
\end{theorem}

\begin{proof}
    Let $\mathcal{S} = \mathcal{S}(D,\mathbf{B})$ be the SFT of the vector presentations of all $(D, \mathbf{B})$-colorings.
    By Theorem \ref{main theorem} all elements of $\mathcal{S}$ are two-periodic.
    %Hence, by Lemma \ref{lemma wang} there is an algorithm to determine whether $\mathcal{S}$ is non-empty.
    %The second part follows by Lemma \ref{lemma effectively}.
    Hence, $\mathcal{S}$ is finite, and the claim follows by Lemma \ref{lemma effectively}.
\end{proof}

\noindent
Corollaries \ref{general square grid1}, \ref{general square grid2}, \ref{general triangular grid1}, \ref{general triangular grid2} and \ref{general king grid} together with above theorem yield the following corollary.

\begin{corollary}
    The following decision problems are decidable for a given matrix $\vec{B}$ satisfying the given conditions.
    \begin{itemize}
        \item The existence of $(D,\mathbf{B})$-colorings where $D$ is the relative 1-neighborhood of the square grid and $\det(\mathbf{B} - \mathbf{I}) \neq 0$.
        \item The existence of $(D,\mathbf{B})$-colorings where $D$ is the relative 1-neighborhood of the triangular grid and $\det(\mathbf{B} + \mathbf{I}) \neq 0$.
        \item The existence of $(D,\mathbf{B})$-colorings where $D$ is the relative $r$-neighborhood of the square grid and $\mathbf{B}$ is arbitrary.
        \item The existence of $(D,\mathbf{B})$-colorings where $D$ is the relative $r$-neighborhood of the triangular grid and $\mathbf{B}$ is arbitrary.
        \item The existence of $(D,\mathbf{B})$-colorings where $D$ is the relative $r$-neighborhood of the king grid and $\det(\mathbf{B}) \neq 0$.
    \end{itemize}
\end{corollary}

\begin{theorem}
    Given a polynomial $f$ in two variables with line polynomial factors in at most one parallel direction there is an algorithm to determine whether there exist any two-dimensional configurations over an alphabet of size $n$ that have low abelian complexity with respect to $f$.
    In fact, there are only finitely many such configurations and one can effectively construct all of them.
\end{theorem}

\begin{proof}
    The set $\{ c \in \{ \vec{e}_1 , \ldots , \vec{e}_n \}^{\Z^2} \mid \mathbf{I} f c = 0 \}$ of the vector presentations of all configurations over an alphabet of size $n$ with low abelian complexity with respect to $f$ is an SFT.
    By Theorem \ref{theorem abelian} it contains only two-periodic configurations and hence it is finite.
    Thus, by Lemma \ref{lemma effectively} we have the claim.
\end{proof}

\section{Conclusions} \label{conclusions}

We studied two-dimensional perfect colorings and proved a general condition (Theorem \ref{main theorem}) for their forced periodicity using an algebraic approach to multidimensional symbolic dynamics.
As corollaries of this theorem we obtained new proofs for known results of forced periodicity in the square and the triangular grid and a new result in the king grid.
Moreover, we generalized a statement of forced periodicity of two-dimensional configurations of low abelian complexity.
Also, some observations of algorithmic decidability were made in the context of forced periodicity.

All our results of forced periodicity of perfect colorings used Theorem \ref{theorem on line polynomial factors} and hence concerned only two-dimensional configurations. However, a 
$d$-dimensional version of  Theorem \ref{theorem on line polynomial factors} 
exists~\cite{kari_DLTinvited}, and so we wonder whether an analogous result to Theorem \ref{main theorem} exists that would give a sufficient condition for forced periodicity of $d$-dimensional perfect colorings for arbitrary dimension $d$.
Note that clearly every one-dimensional perfect coloring is necessarily periodic.

% In fact, there is something we can say about the case for general $d$

% for two-dimensional configurations.
% It would be nice to get a version of Theorem \ref{main theorem} for dimension $d \geq 3$.
% Note that in the one-dimensional setting every perfect coloring is periodic???

% From this Theorem we got as Corollaries new proofs for several known statements of forced periodicity in the square and the hexagonal grid and a new result for forced periodicity in the king grid.

%\newpage

\section*{References}

\bibliographystyle{plain}
\bibliography{Biblio}

\begin{thebibliography}{10}

\bibitem{Axenovich}
M.~A. Axenovich.
\newblock On multiple coverings of the infinite rectangular grid with balls of
  constant radius.
\newblock {\em Discrete Mathematics}, 268(1):31 -- 48, 2003.

\bibitem{ballier}
A.~Ballier, B.~Durand, and E.~Jeandal.
\newblock {Structural aspects of tilings}.
\newblock In Susanne Albers and Pascal Weil, editors, {\em 25th International
  Symposium on Theoretical Aspects of Computer Science}, volume~1 of {\em
  Leibniz International Proceedings in Informatics (LIPIcs)}, pages 61--72,
  Dagstuhl, Germany, 2008. Schloss Dagstuhl--Leibniz-Zentrum fuer Informatik.

\bibitem{bhattacharya}
S.~Bhattacharya.
\newblock Periodicity and decidability of tilings of $\mathbb{Z}^{2}$.
\newblock {\em American Journal of Mathematics}, 142, 02 2016.

\bibitem{tullio}
T.~Ceccherini-Silberstein and M.~Coornaert.
\newblock {\em Cellular Automata and Groups}.
\newblock Springer Monographs in Mathematics. Springer Berlin Heidelberg, 2010.

\bibitem{coveringcodes}
G.~Cohen, I.~Honkala, S.~Litsyn, and A.~Lobstein.
\newblock {\em Covering Codes}.
\newblock Elsevier, 1997.

\bibitem{cox}
D.~A. Cox, J.~Little, and D.~O'Shea.
\newblock {\em Ideals, Varieties, and Algorithms: An Introduction to
  Computational Algebraic Geometry and Commutative Algebra}.
\newblock Springer, 2015.

\bibitem{geravker-puzynina}
N.~Geravker and S.~A. Puzynina.
\newblock Abelian {N}ivat's conjecture for non-rectangular patterns.
\newblock arXiv:2111.04690, December 2021.

\bibitem{godsil}
C.~Godsil.
\newblock Equitable partitions.
\newblock {\em Paul Erdös is Eighty Vol. 1}, pages 173--192, 1993.

\bibitem{greenfeld-tao}
R.~Greenfeld and T.~Tao.
\newblock The structure of translational tilings in $\mathbb{Z}^d$.
\newblock {\em Discrete Analysis}, 2021.

\bibitem{greenfeld-tao2022}
R.~Greenfeld and T.~Tao.
\newblock A counterexample to the periodic tiling conjecture, 2022.

\bibitem{fundamentals-of-domination}
T.~W. Haynes, S.~Hedetniemi, and P.~Slater.
\newblock {\em Fundamentals of Domination in Graphs}.
\newblock CRC Press, 1 edition, 1997.

\bibitem{DLT}
E.~Heikkil{\"a}, P.~Herva, and J.~Kari.
\newblock On perfect coverings of two-dimensional grids.
\newblock In Volker Diekert and Mikhail Volkov, editors, {\em Developments in
  Language Theory}, pages 152--163, Cham, 2022. Springer International
  Publishing.

\bibitem{casurveyjarkko}
J.~Kari.
\newblock Theory of cellular automata: A survey.
\newblock {\em Theoretical Computer Science}, 334(1):3--33, 2005.

\bibitem{surveyjarkko}
J.~Kari.
\newblock Low-complexity tilings of the plane.
\newblock In {\em Descriptional Complexity of Formal Systems - 21st {IFIP} {WG}
  1.02 International Conference, {DCFS} 2019}, volume 11612 of {\em Lecture
  Notes in Computer Science}, pages 35--45. Springer, 2019.

\bibitem{kari_DLTinvited}
J.~Kari.
\newblock Expansivity and periodicity in algebraic subshifts.
\newblock Submitted for publication, 2022.

\bibitem{karimoutot}
J.~Kari and E.~Moutot.
\newblock {Nivat's conjecture and pattern complexity in algebraic subshifts}.
\newblock {\em Theoretical Computer Science}, 777:379 -- 386, 2019.

\bibitem{icalp}
J.~Kari and M.~Szabados.
\newblock An algebraic geometric approach to {N}ivat's conjecture.
\newblock In {\em Proceedings of ICALP 2015, part II}, volume 9135 of {\em
  Lecture Notes in Computer Science}, pages 273--285, 2015.

\bibitem{fullproofs}
J.~Kari and M.~Szabados.
\newblock An algebraic geometric approach to {N}ivat's conjecture.
\newblock {\em Information and Computation}, 271:104481, 2020.

\bibitem{kurka}
P.~Kurka.
\newblock {\em Topological and Symbolic Dynamics}.
\newblock Collection SMF. Soci{\'e}t{\'e} math{\'e}matique de France, 2003.

\bibitem{lagarias-wang}
J.~C. Lagarias and Y.~Wang.
\newblock Tiling the line with translates of one tile.
\newblock {\em Inventiones Mathematicae}, 124:341--365, 1996.

\bibitem{lindmarcus}
D.~Lind and B.~Marcus.
\newblock {\em An Introduction to Symbolic Dynamics and Coding}.
\newblock Cambridge University Press, 1995.

\bibitem{lothaire}
M.~Lothaire.
\newblock {\em Combinatorics on Words}.
\newblock Cambridge Mathematical Library. Cambridge University Press, 2
  edition, 1997.

\bibitem{Moore1962}
E.~F. Moore.
\newblock Machine models of self-reproduction.
\newblock 1962.

\bibitem{morse-hedlund}
M.~Morse and G.~A. Hedlund.
\newblock Symbolic dynamics.
\newblock {\em American Journal of Mathematics}, 60(4):815--866, 1938.

\bibitem{Myhill1963}
J.~R. Myhill.
\newblock The converse of {M}oore’s {G}arden-of-{E}den theorem.
\newblock 1963.

\bibitem{newman}
D.~Newman.
\newblock Tesselation of integers.
\newblock {\em J. Number Theory}, 9(1):107--111, 1977.

\bibitem{Nivat}
M.~Nivat.
\newblock Invited talk at the 24th {I}nternational {C}olloquium on {A}utomata,
  {L}anguages, and {P}rogramming ({ICALP} 1997), 1997.

\bibitem{puzynina}
S.~A. Puzynina.
\newblock Perfect colorings of radius $r > 1$ of the infinite rectangular grid.
\newblock {\em Èlektron. Mat. Izv.}, 5:283--292, 2008.

\bibitem{puzynina2}
S.~A. Puzynina.
\newblock On periodicity of generalized two-dimensional infinite words.
\newblock {\em Information and Computation}, 207(11):1315--1328, 2009.

\bibitem{puzynina3}
S.~A. Puzynina.
\newblock Aperiodic two-dimensional words of small abelian complexity.
\newblock {\em The Electronic Journal of Combinatorics}, 26(4), 2019.

\bibitem{Szegedy1998}
M.~Szegedy.
\newblock Algorithms to tile the infinite grid with finite clusters.
\newblock {\em Proceedings 39th Annual Symposium on Foundations of Computer
  Science (Cat. No.98CB36280)}, pages 137--145, 1998.

\bibitem{wang}
H.~Wang.
\newblock {Proving theorems by pattern recognition -- II}.
\newblock {\em The Bell System Technical Journal}, 40(1):1--41, 1961.

\end{thebibliography}

\end{document}